\documentclass[11pt]{article}%
\usepackage[font=small,format=plain,labelfont=bf,up]{caption}
\usepackage[a4paper,nohead,ignorefoot,%
twoside,scale=0.85,bindingoffset=1.25cm]{geometry}
\usepackage{amssymb,amsfonts,amsmath,amsthm}
\usepackage[]{graphicx}
\usepackage{color}
\graphicspath{}
\newcommand{\texorpdfstring}[2]{#1}   % dummy definition of \texorpdfstring
\newcommand{\url}[1]{#1} % dummy definition of \texorpdfstring
\usepackage{hyperref}%
\definecolor{gray}{rgb}{0.2,0.2,.2}
\hypersetup{%
  unicode=true,          % non-Latin characters in Acrobat’s bookmarks
  colorlinks=true,       % false: boxed links; true: colored links
  linkcolor=gray,          % color of internal links
  citecolor=gray      % color of links to bibliography
}
\DeclareMathOperator{\sgn}{\mathrm{sgn}}

\definecolor{colorBlue}{rgb}{0.,.0,0.825}
\definecolor{colorGreen}{rgb}{0.,0.75,.0}
\definecolor{colorRed}{rgb}{0.99,0.,.0}

\newcommand{\BMHC}{}
\newcommand{\EMHC}{}
\DeclareMathOperator{\mhRe}{Re}

\newcommand{\iu}{\mathtt{i}}

\newcommand{\fspace}[1]{{\mathsf{#1}}}
\newcommand{\fspaceL}{\fspace{L}}

\newcommand{\Rset}{{\mathbb{R}}}
\newcommand{\Zset}{{\mathbb{Z}}}
\newcommand{\Cset}{{\mathbb{C}}}

\newcommand{\ocinterval}[2]{(#1,\,#2]}%
\newcommand{\cointerval}[2]{[#1,\,#2)}%
\newcommand{\oointerval}[2]{(#1,\,#2)}%
\newcommand{\ccinterval}[2]{[#1,\,#2]}%

\newcommand{\ini}{{\rm ini}}

\newlength{\mhpicDwidth}
\newlength{\mhpicDvsep}
\newlength{\mhpicDhsep}

\newlength{\mhpicPwidth}
\newlength{\mhpicPvsep}
\newlength{\mhpicPhsep}

\newlength{\mhpicWhsep}

\newcommand{\pair}[2]{{\left({#1},\,{#2}\right)}}

\newcommand{\at}[1]{{\left({#1}\right)}}

\newcommand{\bat}[1]{{\big(#1\big)}}
\newcommand{\Bat}[1]{{\Big(#1\Big)}}

\newcommand{\triple}[3]{{\left({#1},\,{#2},\,{#3}\right)}}

\newcommand{\ul}[1]{\underline{#1}}

\newcommand{\D}{\displaystyle}

\newcommand{\bigpar}{\par\quad\newline\noindent}

\newcommand{\abs}[1]{\left|{#1}\right|}

\newcommand{\dint}[1]{\,\mathrm{d}#1}

\newcommand{\Om}{{\Omega}}

\newcommand{\eps}{{\varepsilon}}

\newcommand{\ka}{{\kappa}}
\newcommand{\la}{{\lambda}}
\newcommand{\si}{{\sigma}}

\newcommand{\calD}{\mathcal{D}}
\newcommand{\calE}{\mathcal{E}}
\newcommand{\calF}{\mathcal{F}}

\theoremstyle{plain}
\newtheorem{theorem}             {Theorem}[]
\newtheorem{corollary}  [theorem]{Corollary}

\newtheorem{lemma}      [theorem]{Lemma}

\newtheorem*{result*}{Main result}
\theoremstyle{definition}

\newtheorem*{remarks*}{Remarks}
\newtheorem*{remark*}{Remark}
\usepackage{float}
\begin{document}
%
%
% -----------------------------------------------------------------------------
% - Title information
% -----------------------------------------------------------------------------
%
%
\title{Instability of hysteretic phase interfaces \\ in a mean-field model with \BMHC inhomogeneities\EMHC}
\date{\today}
\author{%  
Michael Herrmann\footnote{Technische Universit\"at Braunschweig, Germany, {\tt michael.herrmann@tu-braunschweig.de}}\and
Barbara Niethammer\footnote{Universit\"at Bonn, Germany, {\tt niethammer@iam.uni-bonn.de}}
 }
\maketitle
%
%
% -----------------------------------------------------------------------------
% - Abstract
% -----------------------------------------------------------------------------
%
\begin{abstract}
We study a system of \BMHC non-identical \EMHC bistable particles that is driven by a dynamical constraint and coupled through a non-local mean-field. Assuming piecewise affine constitutive laws we prove the existence of traveling wave solutions and characterize their dynamical stability. Our findings explain the two dynamical regimes for phase interface that can be observed in numerical simulations with different parameters. We further discuss the convergence to a rate-independent model with strong hysteresis in the limit of vanishing relaxation time.
\end{abstract}
%
% 
% -----------------------------------------------------------------------------
% - MSC and keywords
% -----------------------------------------------------------------------------
%
\quad\newline\noindent% 
\begin{minipage}[t]{0.15\textwidth}%
Keywords:
\end{minipage}%
\begin{minipage}[t]{0.8\textwidth}%
\emph{hysteretic phase interfaces},  
\emph{instability of traveling waves},
\\%
\emph{rate-independent evolution of particle systems}
\end{minipage}%
\medskip
\newline\noindent
\begin{minipage}[t]{0.15\textwidth}%
MSC (2010): %
\end{minipage}%
\begin{minipage}[t]{0.8\textwidth}%
34C55,  	% Hysteresis for ordinary differential equations
35C07,  % Traveling wave solutions
70K50,   %	Bifurcations and instability for nonlinear problems in mechanics
74N30 % Problems involving hysteresis in solids
\end{minipage}%
%
%
% -----------------------------------------------------------------------------
% - Table of contents
% -----------------------------------------------------------------------------
%
\setcounter{tocdepth}{3}
\setcounter{secnumdepth}{3}{\scriptsize{\tableofcontents}}
%
%
%
% -------------------------------------------------------------------------------------
\section{Introduction}\label{sect:Intro}
% -------------------------------------------------------------------------------------
%
%
Phase interfaces and free boundary problems are ubiquitous in the sciences and appear naturally in systems that consist of a large number of bistable particles (or units) which are coupled by certain \mbox{interactions}. The underlying model is often a regularization of a diffusive PDE with \mbox{non-monotone} constitutive relation and the phase boundaries satisfy in the sharp interface limit the Stefan \mbox{condition} as well as an additional formula. In the simplest case, the latter represents the classical Maxwell \mbox{construction} and fixes the nonlinearity at the interface to a certain value. The Cahn-Hilliard \mbox{equation} is the most prominent example but there also exist more sophisticated models with \mbox{hysteresis} that distinguish between standing and moving interfaces and allow for both pinning and depinning \mbox{effects}. Important examples are the viscous approximation studied in \cite{NCP91,Plo94,EP04} and the related lattice model in \cite{HH13,HH18}. See also \cite{LM12} for a numerical discussion and \cite{DG17} for the thermodynamical aspects of the different interface rules.
\par
In this paper we study a simpler equation \BMHC which has been introduced in \cite{MT12} 
as a minimal particle model that exhibits non-trivial plastic behavior in the macroscopic limit with slow
loading, see section \ref{sect11} below for more details and \cite{PT00} for a precursory version. Each \EMHC particle is described by a fast scalar gradient ODE with respect to a double-well potential and the interactions are not of diffusive type but imposed by a non-local mean-field that stems from a slowly varying dynamical constraint. A further important ingredient is an inhomogeneity function which models \BMHC parametric fluctuations on the particle level \EMHC and affects the effective dynamics in the scaling limit of vanishing relaxation times. \BMHC The non-local model in our paper \EMHC also leads to hysteretic phase interfaces but their dynamics is not given by free boundary problems with bulk diffusion and Stefan condition. Instead we find a two dimensional but rate-independent limit equation that combines quasi-stationary approximation with a  flow rule for the interface position. Despite its simplicity the model exhibits a rather complex dynamical behaviour depending on the choice of the parameters. Our main result predicts the existence of two different dynamical regimes depending on whether travelling wave solutions with constant interface width are dynamically stable or not. 
%
%
% -------------------------------------------------------------------------------------
\subsection{Mean-field model}
\label{sect11}
% -------------------------------------------------------------------------------------
%
%
In this paper,  we study the dynamical equation
\begin{align}
\label{MicroDynamics}
\tau\, \partial_t{x}\pair{t}{p}=\si\at{t}+\theta\at{p}-H^\prime\bat{x\pair{t}{p}}
\end{align}
with continuous index variable $p\in\ccinterval{0}{1}$,  time $t\geq0$, and state variable $x\in\Rset$.  We further rely on the following constitutive assumptions:
\begin{enumerate}
\item 
$H^\prime$ is the bistable derivative of a double-well potential $H$ as depicted in Figure \ref{Fig:GenericExamples1}.

\item 
The \BMHC inhomogeneity \EMHC function $\theta$ is strictly increasing and satisfies $\int_0^1 \theta\at{p}\dint{p}=0$. 
\item 
The scalar function $\si$ is chosen such that solutions to \eqref{MicroDynamics} satisfy the dynamical side condition
\begin{align}
\label{MicroConstraint}
\int\limits_{0}^1 x\pair{t}{p}\dint{p}=\ell\at{t}
\end{align}
with prescribed function $\ell$.
\item 
The dynamical constraint $\ell$ is continuous,  at least piecewise differentiable,  and slow in the sense that  $\dot\ell \at{t}$ is of order $O\at{1}$ for all $t\geq0$.
\item 
The relaxation time $\tau$ is positive but small, i.e. $0<\tau\ll1$.
\end{enumerate}
\begin{figure}[ht!]%
\centering{%
\includegraphics[width=0.65\textwidth]{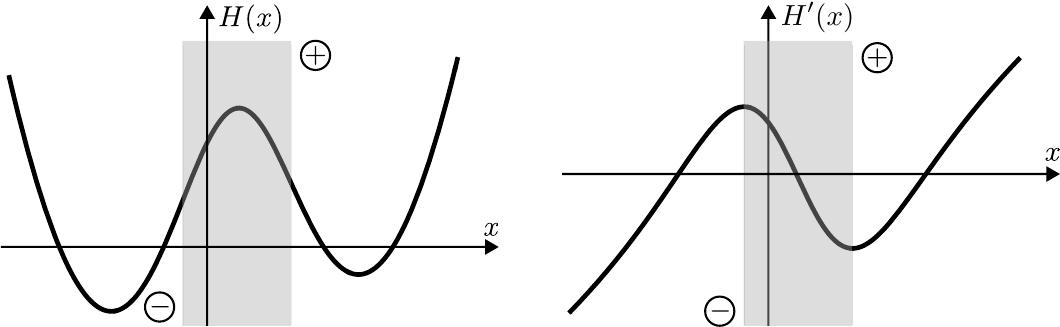}
}%
\caption{%
\emph{Left}.  Graph of a generic double-well potential $H$ which is concave inside the spinodal region (gray box).  The latter separates the two phases (indicated by the encircled signs) in which $H$ is strictly convex.  \emph{Right}.  The corresponding derivative $H^\prime$ is a bistable function as is it consists of two increasing and one decreasing branch.
}%   
\label{Fig:GenericExamples1}%
\end{figure}%
The first assumption is crucial for the existence of phase interfaces while the second one is a mere convention and can be ensured by \BMHC a suitable relabeling of the particles (see the discussion below) and adding constants to $\theta$ and $\si$. \EMHC Direct computations show that \eqref{MicroConstraint} is equivalent to the mean-field formula
\begin{align}
\label{MicroMultiplier}
\si\at{t}=\tau\,\dot\ell\at{t}+\int\limits_{0}^1 H^\prime\bat{x\pair{t}{p}}\dint{p}\,,
\end{align}
which implies that the right hand side in \eqref{MicroDynamics} depends in a non-local and non-autonomous way on the current state of the system and the instantaneous value of $\dot\ell\at{t}$. Finally, the fourth and the fifth assumption guarantee a clear separation between the two time scales in the problem.
\par
The dynamical system \eqref{MicroDynamics}+\eqref{MicroMultiplier} models an infinite ensemble of bistable particles which try to \mbox{minimize} their individual energies. The particles, however, are different thanks \BMHC to the \BMHC inhomogeneities \EMHC imposed by $\theta$ and also coupled by the non-local constraint \eqref{MicroConstraint}. The effective energy of particle $p$ is given by the tilted potential 
\begin{align*}
x\mapsto H\at{x}-\at{\theta\at{p}+\si\at{t}}\,x\,,
\end{align*}
which varies in time (hopefully slowly) and might exhibit either one or two wells depending on the values of $\theta\at{p}$ and $\si\at{t}$.  The energetic balance of the entire ensemble reads
\begin{align}
\label{EnergyBalance}
\tau\,\dot{\calE}\at{t}=\tau\,\si\at{t}\,\dot\ell\at{t}-\calD\at{t}\,,
\end{align}
where 
\begin{align*}
\calE\at{t}=\int\limits_0^1 H\bat{x\pair{t}{p}}\BMHC \dint{p}\EMHC -\int\limits_0^1 \theta\at{p}\,x\pair{t}{p}\BMHC \dint{p}\EMHC\quad\text{and}\quad
\calD\at{t}=\int\limits_0^1\Bat{\theta\at{p}+\si\at{t}-H^\prime\bat{x\pair{t}{p}}}^2\dint{p}
\end{align*}
quantify the total energy and the dissipation of the \BMHC system, respectively. The energy is not neccesarily \mbox{decreasing} due to the evolving side condition but its value at time $t$ depends neither on the \mbox{constraint} $\ell\at{t}$ nor the multiplier $\si\at{t}$. Alternatively one might state \eqref{EnergyBalance} as $\tau\,\dot\calF\at{t}=\tau\,\ell\at{t}\,\dot\si\at{t}-\calD\at{t}$, but the augmented energy $\calF\at{t}:=\calE\at{t}-\si\at{t}\,\ell\at{t}$ depends on $\ell\at{t}$ and in view of \eqref{MicroMultiplier} also on $\dot\ell\at{t}$.
\par
Notice also that equations \eqref{MicroDynamics}+\eqref{MicroMultiplier} are invariant under arbitrary rearrangements with respect to the index variable $p$ since the particles are not coupled by neighbor interactions but solely via the mean-field $\sigma$. The monotonicity of $\theta$ just means that the particles are ordered according to the strength of the inhomogeneity
and this is a convenient setting for the study of phase interfaces. However, in principle one can consider arbitrary functions $\theta$ and
interpret the inhomogeneities as random particle fluctuations that are independent of time (quenched disorder).
\bigpar
The analysis in this paper is restricted to the affine function
\begin{align}
\label{DefTheta}
\theta\at{p}=\delta\at{p-\tfrac12}
\end{align}
\BMHC with slope parameter $\delta>0$ \EMHC and the trilinear constitutive relation
\begin{align}
\label{DefNonl}
H^\prime\at{x}=\left\{\begin{array}{rl}
\!\!x+1&\!\!\text{for}\;\; x\in\ocinterval{-\infty} {-\ka}\,,
\smallskip\\%
\!\!\D-\frac{1-\ka}{\ka}\,x&\!\!\text{for}\;\; x\in\ccinterval{-\ka}{+\ka}\,,
\smallskip\\%
\!\!x-1&\!\!\text{for}\;\;x\in\cointerval{+\ka}{+\infty}\,,
\end{array}\right.
\end{align}
\BMHC which involves the spinodal parameter $0<\ka<1$. This function $H^\prime$ is illustrated in Figure \ref{Fig:Trilinear} and \EMHC corresponds to the piecewise quadratic double-well potential
\begin{align}
\notag%\label{DefPot}
H\at{x}=\frac12\left\{\begin{array}{rl}
\!\!\at{x+1}^2&\!\!\text{for}\;\; x\in\ocinterval{-\infty} {-\ka}\,,
\smallskip\\%
\!\!\D\frac{1-\ka}{\ka}\,\at{\ka-x^2}&\!\!\text{for}\;\; x\in\ccinterval{-\ka}{+\ka}\,,
\smallskip\\%
\!\!\at{x-1}^2&\!\!\text{for}\;\;x\in\cointerval{+\ka}{+\infty} \,.
\end{array}
\right.
\end{align}
\begin{figure}[ht!]%
\vspace{-\medskipamount}
\centering{%
\includegraphics[width=0.35\textwidth]{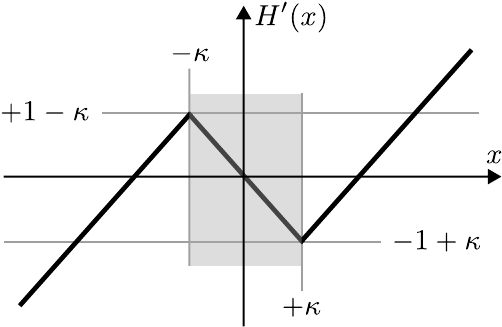}
}%
\caption{%
The trilinear function \eqref{DefNonl} with parameter $0<\ka<1$.  In the degenerate limit $\ka\to0$, the spinodal interval collapses to a single point,  so $H^\prime$ becomes bilinear and discontinuous. For $\ka\geq1$, there is no spinodal region since $H^{\prime\prime}$ is a non-negative function.
}%
\label{Fig:Trilinear}%
\end{figure}%
\BMHC In what follows we \EMHC argue that the effective dynamics for small $\tau$ depends crucially on $\delta$, the strength of the \BMHC inhomogeneities\EMHC. The key argument is the stability or instability of traveling wave solutions which we construct and investigate for the simplified model with \eqref{DefTheta}+\eqref{DefNonl} only. However, we expect that the key findings \BMHC apply to more general double-well potentials and other functions $\theta$ as well, see the discussion in section \ref{sect:existenceTW}. We further \EMHC emphasize that the limiting cases $\ka=0$ and $\ka=1$ are degenerated. For $\ka=1$, equation \eqref{DefNonl} implies the monotonicity of $H^\prime$  and phase interfaces do  hence not exist. For $\ka=0$, however, the width of any phase interface vanishes and cannot alter in time.
\bigpar
Closely related to our work is \cite{MT12}, which studies a finite-dimensional analogue to our \mbox{dynamical} mean-field model consisting of $N$ particles. In that paper, \BMHC the index variable $p$ attains the discrete values $p_k=\eps\,k$ with $k=1,...,N$ and $\eps=1/N$, so the formulas for $\ell$ and $\si$ involve sums instead of integrals. Each particle is interpreted as an \mbox{overdamped} snap-spring system whose elastic energy is given by the combination of $H$ and a parameter $\theta_k$, which
is considered as a static random variable and models quenched disorder. Moreover, the dynamical constraint is related to a time-dependent Dirichlet condition in a hard loading device. \EMHC The analytical perspective, however, differs from our setting. In \cite{MT12}, the authors do not study the propagation of phase interface explicitly but concentrate on the variational aspects of the macroscopic \mbox{continuum} limit and prove that the finite-dimensional gradient flow $\Gamma$-converges to a rate-independent ERIS model provided that first $\tau$ and afterwards $\eps$ tends to 0. The rigorous results in \cite{MT12} allow for \BMHC more general functions \EMHC $\theta$ but assume a bilinear constitutive relation, i.e. $H^\prime$ is given by \eqref{DefNonl} with $\ka=0$. The spatial or temporal width of phase interfaces can hence not oscillate for $\tau>0$ but the rate-independent limit model is more general than the equations below.  
\par
The \mbox{wiggly} \mbox{energy} landscape of snap-spring systems as well as the expected hysteresis in the \mbox{quasi-stationary} regime are also studied in \cite{FZ92,PT00,PT02,PT05} and \cite{Mie11b,MRS12} present a \mbox{general} framework for the derivation of rate-independent limit models from microscopic \mbox{gradient} flows. See also \cite{BS96,Mie05,Mie11a,MR15} for more details \mbox{concerning} the underlying \mbox{mathematical} \mbox{theory} and the admissible solution concepts for rate-independent evolution. Moreover, piecewise \mbox{linear} \mbox{constitutive} relations have also been used \BMHC in \cite{TV05,TV08,ET10,TV10a,TV10b,VK12,HSZ12,GTV22} to characterize phase transition waves and hysteretic phenomena in Hamiltonian particle systems or related latttice models.\EMHC
\par
The non-local coupling of bistable particles can also explain the hysteretic phenomena in \mbox{Lithium-ion} batteries. For instance, the model in \cite{DGH11} describes the entropic effects by white noise instead of \BMHC parametric inhomogeneities \EMHC and can hence be regarded as a family of identical stochastic ODEs which is driven by a dynamical constraint similar to \eqref{MicroConstraint}. The effective  behavior is governed by a \mbox{non-local} variant of the Fokker-Planck equation for the probability distribution and has been \mbox{studied} in \cite{HNV12,HNV14} for different scaling relations between the relaxation time and the strength of the noise.
\par
\BMHC We finally emphasize that the finite- and infinite-dimensional versions of the mean-field behave similarly but a continuous index variable  is more convenient for analytical purposes since a phase interface corresponds to a smooth curve and does not propagate by small jumps in $p$-direction. Our numerical simulations, however, are computed with $N=1000$ particles, where we used an explicit Euler step for the straight forward time discretization. \EMHC
%
%-------------------------------------------------------------------------------------------
\subsection{\texorpdfstring{Formation and propagation of phase interfaces for $\tau>0$ and $\tau=0$}{formation and propagation of phase interfaces }}
%-------------------------------------------------------------------------------------------
%
\BMHC 
The formulas \eqref{DefTheta}+\eqref{DefNonl} as well as our assumptions on the constraint $\ell$ imply --- for any admissible choice of the parameters $\tau$, $\delta$, and $\kappa$ --- that the initial value problem to the dynamical system \eqref{MicroDynamics}+\eqref{MicroMultiplier} is globaly well-posed since the right hand side is Lipschitz continuous with respect to the state variable $x\in\fspaceL^\infty\at{\ccinterval{0}{1}}$.\EMHC
\par
Using standard arguments for parametrized ODEs we easily show that the set of all states being continuous and stricly increasing with respect to $p$ is invariant under the dynamics and numerical simulations show that even the solutions with random initial data belong to this class after a certain transient time $t_*$.  In other words, equations \eqref{MicroDynamics}+\eqref{MicroMultiplier} have a strong tendency to produce single-interface solutions as sketched in Figure~\ref{Fig:GenericExamples2}, where $\xi_-\at{t}$ and $\xi_+\at{t}$ denote the values of $p$ at which the solution enters and leaves the spinodal region, respectively. In particular, for $\tau>0$ and all times $t\geq t_*\geq 0$ we have 
\begin{align}
\label{DefPhases}
\begin{split}
x\pair{t}{p}\in\left\{\begin{array}{rcl}
\ocinterval{-\infty}{-\ka}&\text{for}& p\in\ccinterval{0}{\xi_-\at{t}}
\smallskip\\%
\ccinterval{-\ka}{+\ka}&\text{for}&p\in\ccinterval{\xi_-\at{t}}{\xi_+\at{t}}
\smallskip\\%
\cointerval{+\ka}{+\infty}&\text{for}&p\in\ccinterval{\xi_+\at{t}}{1}
\end{array}
\right. 
\end{split}
\end{align}
\BMHC with $0\leq\xi_-\at{t}\leq\xi_+\at{t}\leq1$.\EMHC
\begin{figure}[ht!]%
\centering{%
\includegraphics[width=0.35\textwidth]{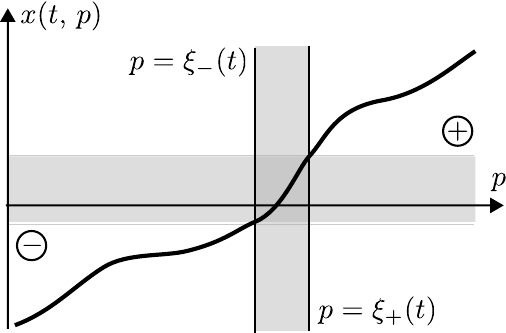}
}%
\caption{%
Cartoon of a continuous monotone state with two interface positions $\xi_-\at{t}$ and $\xi_+\at{t}$. }%
\label{Fig:GenericExamples2}%
\end{figure}%
\par\noindent
The key features of the particle dynamics can be summarized as follows. As long as $x\pair{t}{p}$ belongs to one of the stable regions, the local dynamics is governed by either one of the stable ODEs 
\begin{align*}
\tau\,\partial_t x\pair{t}{x} = \si\at{t}+\delta\,\at{p-\tfrac12}-1-x\pair{t}{p}
\quad\text{as long as}\quad x\pair{t}{p}<-\kappa
\end{align*}
or
\begin{align*}
\tau\,\partial_t x\pair{t}{x} = \si\at{t}+\delta\,\at{p-\tfrac12}+1-x\pair{t}{p}
\quad\text{as long as}\quad x\pair{t}{p}>+\kappa\,,
\end{align*}
which describe that $x\pair{t}{p}$ wants to approach $\si\at{t}+\delta\,\at{p-\tfrac12}\mp1$ rather quickly provided that the mean-field $\si\at{t}$ changes slowly. Inside the spinodal region, however, particle $p$ evolves according to the unstable ODE
\begin{align*}
\tau\,\partial_t x\pair{t}{x} = \si\at{t}+\delta\,\at{p-\tfrac12}+\frac{1}{\ka}x\pair{t}{p}
\quad\text{as long as}\quad -\ka<x\pair{t}{p}<+\kappa
\end{align*}
so we expect that the particle leaves the spinodal region after a short period of time with large exit velocity. The interplay between the stable and the unstable ODEs ensure that there exist basically the following three different modes for the dynamics of the phase interface.
\begin{enumerate}
\item 
$\dot\xi_-\at{t}>0$ and $\dot{\xi}_+\at{t}<0$\,: The width of the interface shrinks since particles leave the spinodal region on both sides.
\item 
$\dot\xi_-\at{t}<0$ and $\dot{\xi}_+\at{t}<0$\,: The phase interface moves to the left (i.e., into the negative phase $x<-\ka$) since the particles enter$\big|$leave the spinodal region on the left$\big|$right hand side. A~particle in front of the interface enters the spinodal region slowly, gets strongly accelerated during its short spinodal visit, and relaxes quickly to a slow motion in the back of the interface. 
\item 
$\dot\xi_-\at{t}>0$ and $\dot{\xi}_+\at{t}>0$\,: The  interface propagates into the phase $x>+\kappa$, i.e. to the right.
\end{enumerate}
In summary, the non-local dynamical system \eqref{MicroDynamics}+\eqref{MicroMultiplier}+\eqref{DefTheta}+\eqref{DefNonl} allows for both standing and moving phase interfaces. Each interface has a small but positive width and a moving one also exhibits a tail in its back which stems from the quick deceleration of fast particles leaving the spinodal region. See also the numerical simulations in Figures \ref{Fig:TWForm1} and \ref{Fig:TWForm2} below. 
\par%
On a formal level we expect that the limit dynamics for $\tau=0$ can be described by slowly varying quasi-stationary states which do not penetrate the spinodal region but exhibit a jump discontinuity which is located at $x=\xi\at{t}=\xi_\pm\at{t}$ and represents a phase interface of vanishing width. More precisely, for $\tau=0$ the state of the particle system is at any time $t\geq0$ given by
\begin{align}
\label{LimStates}
x\pair{t}{p}=\si\at{t}+\delta\,\at{p-\tfrac12}+\left\{
\begin{array}{ccc}
-1&\text{for}&0\leq p<\xi\at{t}\,,
\smallskip\\%
+1&\text{for}&\xi\at{t}<p\leq 1\,,
\end{array}
\right.
\end{align}
and this implies 
\begin{align}
\label{LimEqnOfState}
\si\at{t}+1-2\,\xi\at{t}=\ell\at{t}
\end{align}
as equation of state between $\si$, $\xi$, and $\ell$. However, this algebraic equation between the input $\ell$ and the two output quantities $\si$ and $\xi$ must be accompanied by a dynamical flow rule that relates interface motion to the instantaneous quasi-stationary state and the dynamical multiplier. In view of the heuristic arguments from above concerning the slow and the fast processes for $\tau>0$ and motivated by numerical simulations (see for instance Figures \ref{Fig:PerFor1} and \ref{Fig:PerFor2} below) we expect that left and right moving interfaces comply with
\begin{align}
\label{LimFlowRule}
\begin{split}
\dot\xi\at{t}>0\quad\implies\quad \si\at{t}+\delta\,\at{\xi\at{t}-\tfrac12}=+1-\ka
\,,\\%
\dot\xi\at{t}<0\quad\implies\quad \si\at{t}+\delta\,\at{\xi\at{t}-\tfrac12}=-1+\ka\,,
\end{split}
\end{align}
\BMHC while standing interfaces for $\tau=0$ are characterized by \EMHC 
\begin{align*}
\BMHC\si\at{t}+\delta\,\at{\xi\at{t}-\tfrac12}\in\oointerval{-1+\ka}{+1-\ka}\quad \implies 
\quad \dot\xi\at{t}=0\,,\quad \dot\si\at{t}=\dot\ell\at{t}\,.\EMHC 
\end{align*}
\begin{figure}[ht!]%
\centering{%
\includegraphics[width=0.95\textwidth]{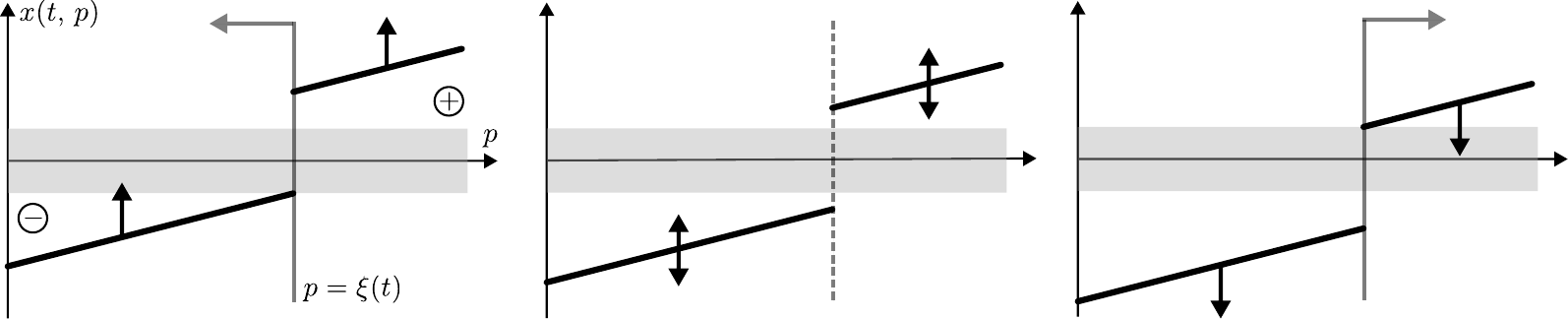}
}%
\caption{%
Piecewise affine states with slope $\delta$ for left-moving,  standing,  and right-moving interfaces as predicted by the limit model via \eqref{LimStates}. The jump at the phase interface has always height $2$ while the left- and right-sided limits comply with the hysteretic flow rule \eqref{LimFlowRule}.}%
\label{Fig:LimitStates}%
%%\end{figure}%
%%%
%%%
%%\begin{figure}[ht!]%
\medskip
\centering{%
\includegraphics[width=0.95\textwidth]{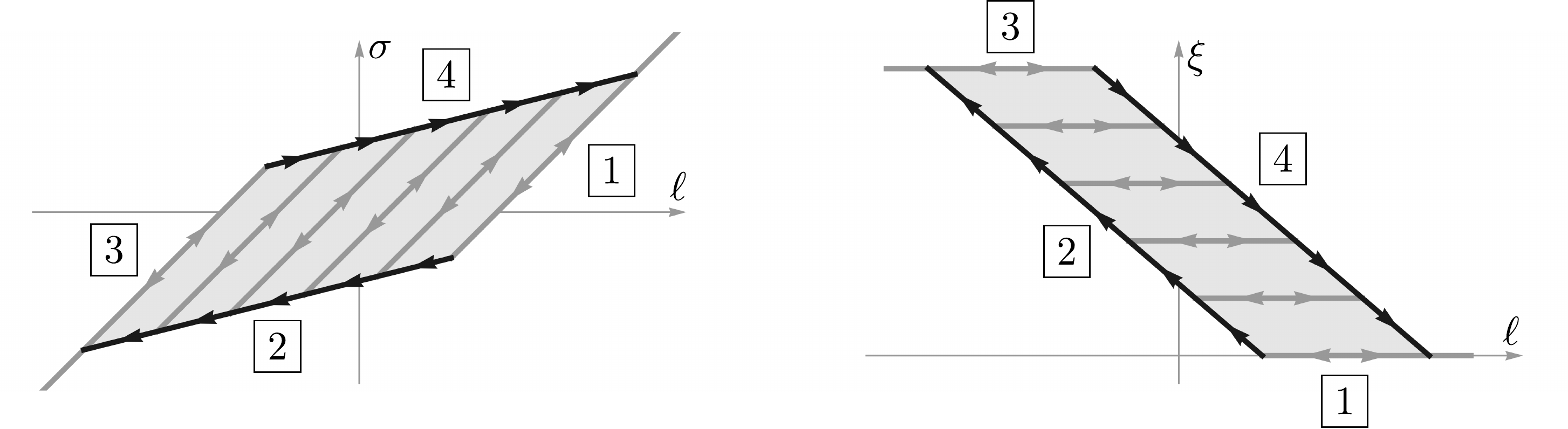}
}%
\caption{%
Cartoon of the rate-independent limit dynamics in two different planes, shown for $\delta=\tfrac23$ and $\ka=\tfrac13$.  Standing and moving interfaces are represented by gray and black lines, respectively, and the boundary of the hysteresis loops is described by the affine equations in \eqref{HysteresisLoops}. Given $\ell$, the evolution of $\si$ and $\xi$ can be deduced by following the admissible directions in the diagrams.
}%
\label{Fig:LimitDynamics}%
\end{figure}%
In particular, the phase interface can only propagate to the left$\big|$right, if the left-sided$\big|$right-sided limit of the quasi-stationary state \eqref{LimStates} equals ${-}\ka\big|{+}\ka$, so particles in front of the moving interface enter the spinodal region from below$\big|$above. Moreover,
 the interface speed as well as the temporal change of the mean-field are completely determined by the dynamical constraint since the combination of \eqref{LimEqnOfState}+\eqref{LimFlowRule} ensures
\begin{align*}
\dot\xi\at{t}\neq0\quad\implies\quad \dot\xi\at{t}=-\frac{1}{2+\delta}\,\dot\ell\at{t}\,,\quad \dot\si\at{t}=+\frac{\delta}{2+\delta}\,\dot\ell\at{t}\,.
\end{align*}
In conclusion, the candidate for the limit dynamics is a two-dimensional and rate-independent \mbox{dynamical} systems for $\si$ and $\xi$ that is driven by the \mbox{non-local} constraint $\ell$ and exhibits \mbox{hysteresis} phenomena, where the parameters $\ka$ and $\delta$  appear in the flow rule for the interface \eqref{LimFlowRule} but not in the equation of state \eqref{LimEqnOfState}. Figure \ref{Fig:LimitStates} illustrates the consistent quasi-stationary states while the effective \mbox{evolution} is \mbox{depicted} in Figure \eqref{Fig:LimitDynamics}. The boundary of the admissible set can be computed by means of  \eqref{LimStates}+\eqref{LimEqnOfState}+\eqref{LimFlowRule} and consists of the following four parts:
\begin{align}
\label{HysteresisLoops}
\begin{array}{cclcl}
\text{\emph{part}\;}&&\text{\emph{type of interface}}&&\text{\emph{relations between $\si$, $\xi$, and $\ell$}}
\smallskip\\\hline\vspace{-2\smallskipamount}\\
\boxed{1}&&\text{standing}&&\sigma=\ell-1\,,\;\;\xi=0
\smallskip\\%
\boxed{2}&&\text{right moving}&&\at{1+\tfrac12\,\delta}\,\sigma+\tfrac12\,\delta\,\ell=-1+\ka=\at{2+\delta}\,\at{\xi-\tfrac12}+\ell
\smallskip\\%
\boxed{3}&&\text{standing}&&\sigma=\ell+1\,,\;\; \xi=1
\smallskip\\%
\boxed{4}&&\text{left moving}&&\at{1+\tfrac12\,\delta}\,\sigma+\tfrac12\,\delta\,\ell=+1-\ka=\at{2+\delta}\,\at{\xi-\tfrac12}+\ell
\end{array}
\end{align}
%\vspace{-2\medskipamount}%
%
%
%
%
\BMHC We finally mention that hysteretic interface rules for piecewise continuous functions appear in other models as well. For instance, \cite{Peg87} considers the viscous regularization of an ill-posed wave \mbox{equation} with bilinear consitutive \mbox{relation}, \mbox{derives} admissible jump conditions in the limit of \mbox{vanishing} \mbox{viscosity}, and studies the implications for slow load-deformation experiments.  \EMHC
%
%
% -------------------------------------------------------------------------------------
\subsection{Main result and its interpretation}
%\label{sect13}
% -------------------------------------------------------------------------------------
%
From a mathematical point of view, it is quite natural to ask whether the solutions to the particle system converge as $\tau\to0$ to a trajectory of the limit model and how a rigorous proof could be accomplished. A first promising ingredient is the energy balance \eqref{EnergyBalance} because it ensures that the dissipation is small at most of the times and one can show that this implies that the interface width $\xi_+\at{t}-\xi_-\at{t}$ is sufficiently small as well as the asymptotic validity of the equation of state \eqref{LimEqnOfState}. To complete the proof one would like to establish uniform regularity estimates for $\dot{\si}$ and $\dot{\xi}_\pm$ because this would guarantee that the hysteretic flow rule  \eqref{LimFlowRule} is satisfied in the limit. 
\par
However, numerical simulations indicate that both the mean-field and the interface position do not always evolve very regularly. It might happen that the particle model does in fact evolve as predicted by the formal limit equations, but it is also possible that the macroscopic quantities and their derivatives exhibit strong temporal oscillations. The first and the second dynamical regime is illustrated in Figure \ref{Fig:PerFor1} and \ref{Fig:PerFor2}, respectively, where we always used the same function $\ell$ and compare simulations with different choices of $\delta$. In particular, the oscillations in the second simulation are neither caused by the dynamical constraint nor by the choice of $\ka$ but reflect an internal oscillatory mode which is not present in the first example. 
\begin{figure}[ht!]%
\centering{%
\includegraphics[width=0.95\textwidth]{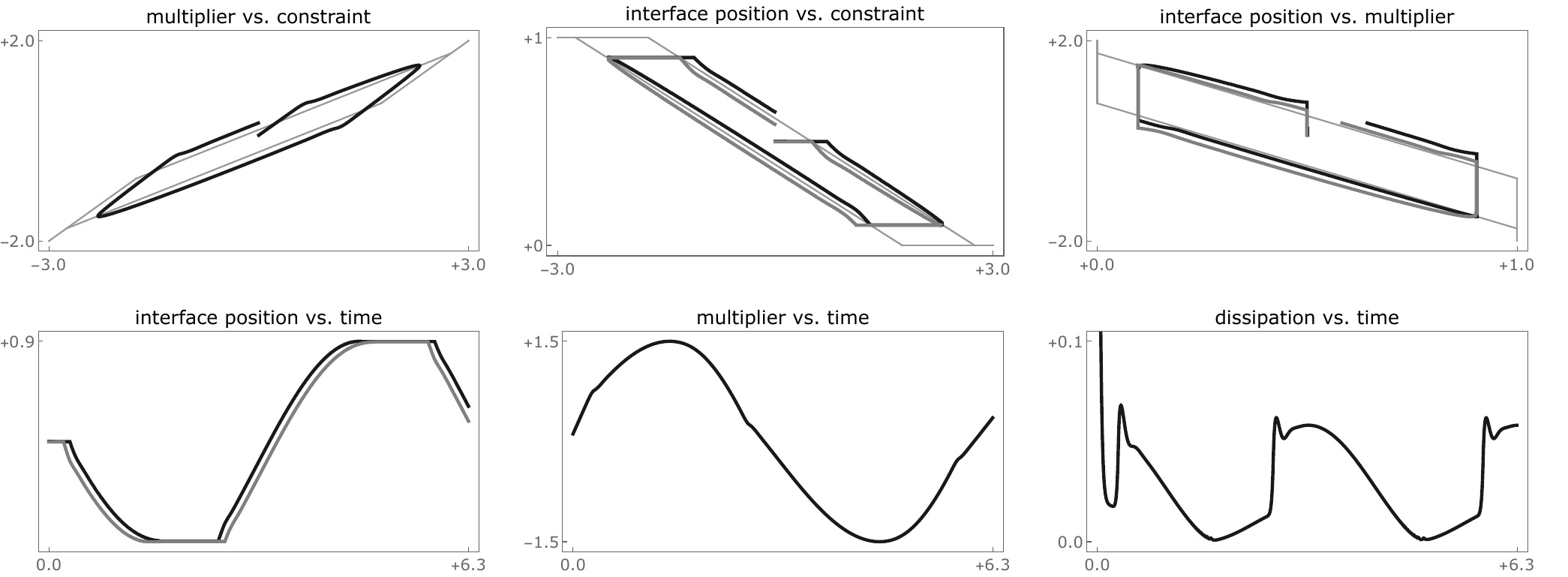}%
}%
\caption{%
Numerical simulation with initial data $x_{\ini}\at{p}=\sgn\at{p-0.5}$, parameters $\delta=2.5$, $\ka=0.5$, $\tau=0.05$, and periodic forcing  $\ell\at{t}=\sin\at{t}$ on the time interval $\ccinterval{0}{2\,\pi}$.  \emph{Top}. Projection of the dynamical trajectory to three different planes, where the thin lines in the background represent the limit model as shown in Figure \ref{Fig:LimitDynamics}. In the last two columns, the interface position is defined by $\xi_-$ (gray) or $\xi_+$ (black). \emph{Bottom}.  Evolution of $\xi_-$, $\xi_+$, $\si$, and $\calD$. \emph{Interpretation}. The particle systems follows for small $\tau$ the limit model since left and right moving traveling wave solutions are dynamically stable. }%
\label{Fig:PerFor1}%
%\end{figure}%
%
%
\bigskip
%
%\begin{figure}[ht!]%
\centering{%
\includegraphics[width=0.95\textwidth]{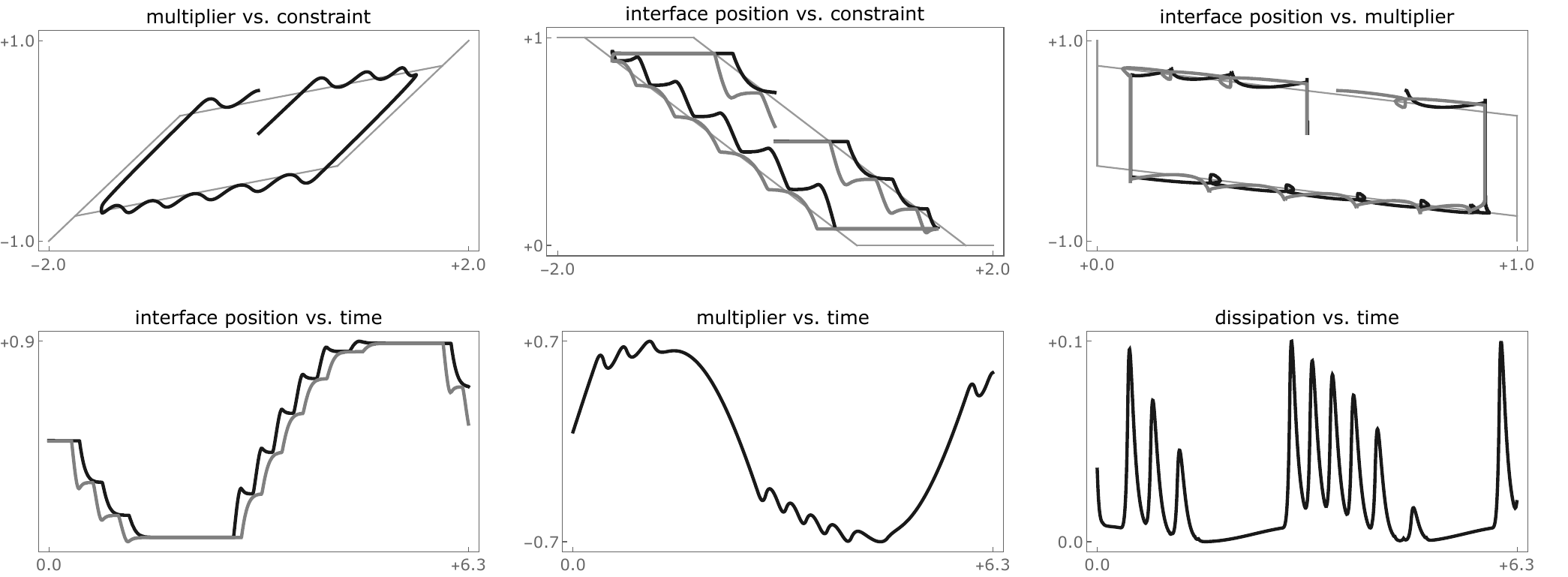}%
}%
\caption{%
Recomputation of the solution from Figure \ref{Fig:PerFor1} with modified parameter $\delta=0.5$. \emph{Interpretation}. The strong temporal oscillations reflect the loss of stability for traveling waves and the deviations from the predictions of the limit model are rather strong.
}%
\label{Fig:PerFor2}%
\end{figure}%
\par
In this paper we offer the following explanation for the existence of the two different regimes.
\begin{result*}
The dynamical system \eqref{MicroDynamics}$+$\eqref{MicroMultiplier}$+$\eqref{DefTheta}$+$\eqref{DefNonl} admits for any $\tau>0$ traveling wave solutions provided that the dynamical constraint changes almost linearly. These waves propagate with constant interface speed but can be stable or unstable depending on the parameters $\tau$, $\kappa$, and $\delta$. In particular, they are unstable for $\delta<2$ and arbitrary $0<\ka<1$ if $\tau$ is sufficiently small.
\end{result*}
In fact, it seems that the numerical data in the non-oscillatory regime can be approximated by a slowly modulated traveling wave while this is not the case in the second regime, where the width of the interface oscillates as well. We are not able to prove this approximation result rigorously but mention that the slow modulation concerns both the interface width and the wave speed which slowly adjust to the unique values that are compatible with the current change of the dynamical constraint. See also Figures \ref{Fig:TWForm1} and \ref{Fig:TWForm2} below for related simulations with linearly increasing constraint and notice that the different dynamical behaviour concerns times at  which the limit model predicts a moving phase interface. Otherwise the width of the interface is always negligible and the corresponding state of the particle system is very close to the quasi-stationary  approximation \eqref{LimStates}.
\par
We further emphasize that the instability of traveling waves does not imply the \mbox{invalidity} of the rate-independent limit model since the amplitudes in the oscillations of $\si$ and $\xi_\pm$ seem to disappear in the limit $\tau\to0$. It only indicates that the rigorous passage to the limit $\tau\to0$ is rather delicate in the case of small \BMHC inhomogeneities \EMHC and non-vanishing spinodal region. More precisely, for $\delta>2$ we expect that the mean-field $\si$, the interface positions $\xi_\pm$, and their derivatives change slowly and converge for $\tau\to0$ to pointwise limits that satisfy the equation of state \eqref{LimEqnOfState} as well as the hysteretic flow rule \eqref{LimFlowRule} for all times $t\geq0$. In particular, changes in the propagation mode of the interface only happen when predicted by the limit model and are solely enforced by the dynamical constraint. For $0<\delta<2$, however, we observe rapid sign changes in the derivatives of $\xi_\pm$ and $\si$ because the width of the interface can constantly switch between expansion and contraction. \BMHC These fluctuations complicate the mathematical analysis and must be taken into account in any convergence proof for $\tau\to0$. For small $\delta$, one has to work with weak instead of strong convergence and must control the evolution of the phase interface more carefully. \EMHC 
\par
It would be desirable to identify the next order corrections and to derive a more sophisticated model for the small-parameter dynamics of \BMHC \eqref{MicroDynamics}$+$\eqref{MicroMultiplier} that accounts for a breathing interfaces but this lies beyond the scope of this paper. It has already been pointed out in \cite{MT12} that the \mbox{temporal} \mbox{oscillations} in mean-field models can be linked to the Neishtadt phenomenon (see for \mbox{instance} \cite{Nei09}), so the \mbox{theory} of \mbox{dynamical} bifurcations might provide asymptotic formulas for the expected \mbox{amplitudes} and \mbox{frequencies} of the deviations from the limit model. \EMHC
\par
\BMHC We finally remind \EMHC that the limiting case $\ka=0$ is degenerate. The right hand side of \eqref{MicroDynamics} is not continuous, so the corresponding initial value problem is ill-posed and admits multiple solutions. There still exist traveling wave solutions as given below but their stability can not be related to a linearized equation and its spectral properties. In particular, our results do not apply to the system with bilinear constitutive law and one might even argue that this should behave much nicer since the width of the phase interfaces cannot oscillate at all.
\bigpar
The paper is organized as follows. We first characterize traveling waves by combining piecewise ODE arguments with natural matching conditions. This gives rise to the explicit formulas in \mbox{Theorem} \ref{Thm:TW} and yields with \eqref{Lem:TW:Eqn4}+\eqref{TWConstraint} consistent formulas for the interface width and the \mbox{dynamical} constraint. Afterwards we study the linearized dynamical equation in the comoving frame as well as the related eigenvalue problem \eqref{LinStab.EP}. The latter consists of a family of linear ODEs with piecewise constant coefficients which are still coupled by a \mbox{mean-field}. In Theorem \ref{Thm:Spectrum} we relate the entire spectrum to the zeros of certain \mbox{transcendental} \mbox{functions} which depend on $\tau$, $\ka$ and $\delta$. For fixed $\tau>0$, the \mbox{implied} stability properties can only be studied numerically by plotting the zeros in the complex plane but there exist simplified  expressions for the limit $\tau\to0$ as derived in Corollary~\ref{Cor:Instab}.
%
%
%-------------------------------------------------------------------------------------------
\section{Traveling waves and their stability}
%-------------------------------------------------------------------------------------------
%
%-------------------------------------------------------------------------------------------
\subsection{Numerical examples of interface propagation} %\label{sect:NumTW}
%-------------------------------------------------------------------------------------------
%
%
\begin{figure}[ht!]%
\centering{%
\includegraphics[width=0.95\textwidth]{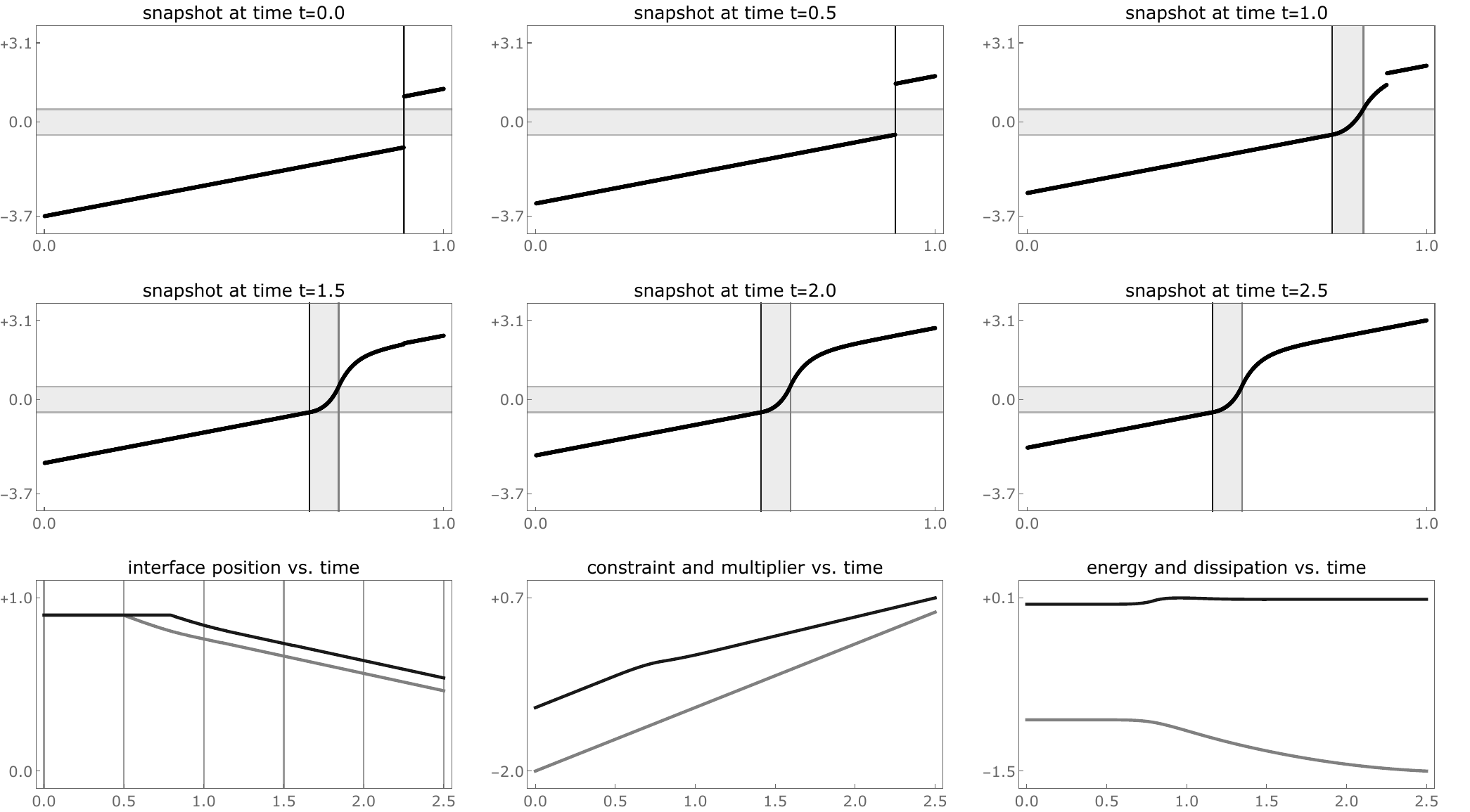}%
}%
\caption{%
Numerical solution in the simplified setting --- see  \eqref{NormalizedConstraint} and \eqref{WellPreparedData}) --- with $\ka=0.5$, $\delta=3.0$, and $\tau=0.2$. The first two rows show the graph of $x\pair{t}{\cdot}$ at six equidistant times, while the third row illustrates the evolution of $\xi_-\at{t}$, $\ell\at{t}$, $\calE\at{t}$ (gray) as well as $\xi_+\at{t}$, $\si\at{t}$, $\calD\at{t}$ (black). 
}%
\label{Fig:TWForm1}%
%\end{figure}%
%%
\bigskip
%%
%\begin{figure}[ht!]%
\centering{%
\includegraphics[width=0.95\textwidth]{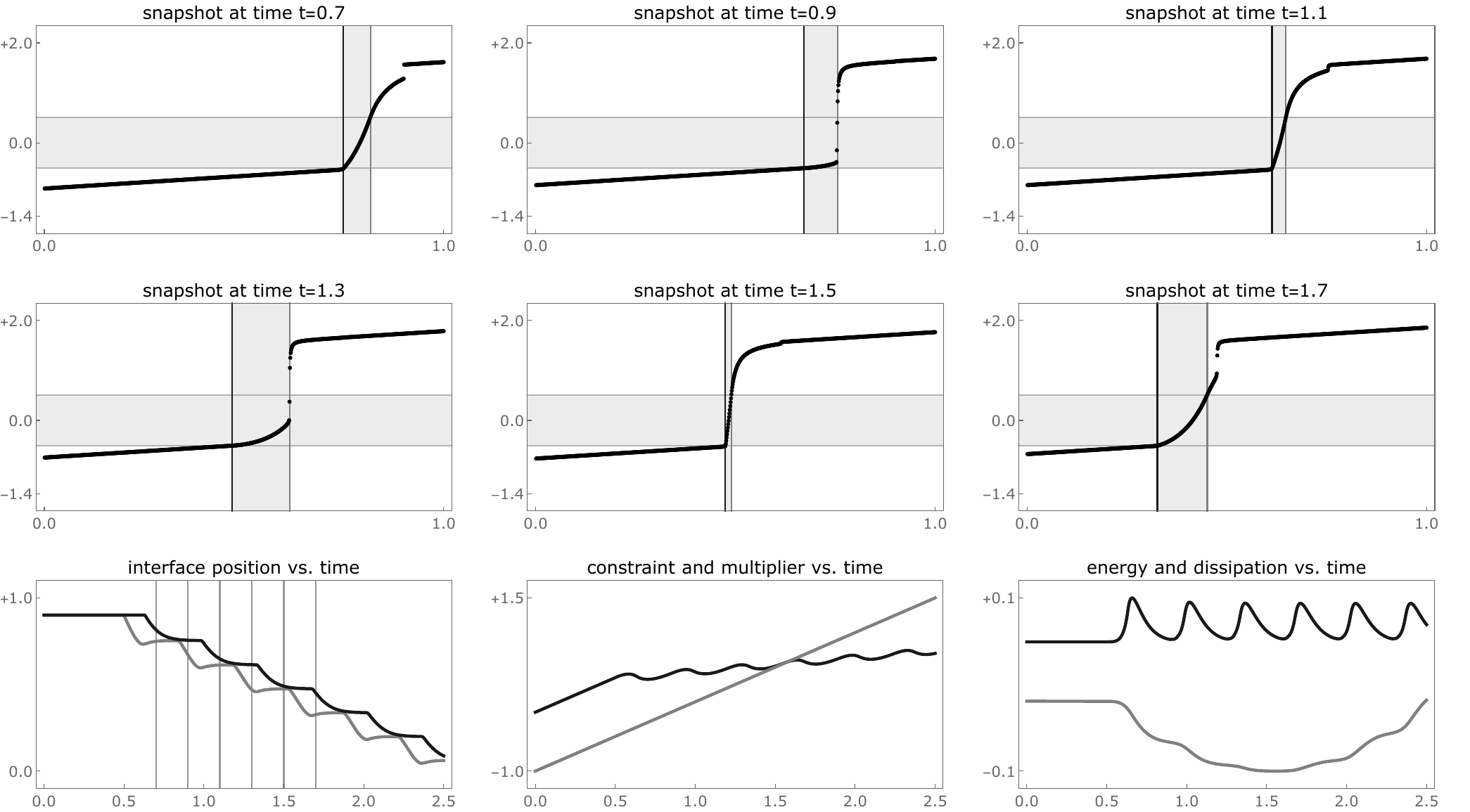}
}%
\caption{%
The simplified setting with $\ka=0.5$, $\delta=0.5$,  and $\tau=0.05$.  The interface width now exhibits rather strong temporal oscillations. See also Figure \ref{Fig:TWForm3} for another choice of $\tau$.
}
\label{Fig:TWForm2}%
\end{figure}%
\noindent
We already mentioned in \S\ref{sect:Intro} that moving phase interfaces can behave rather differently depending on whether a constant interface width is stable or not. To elucidate the key dynamical features we study a simplified setting with
\begin{enumerate}
\item 
the constantly increasing dynamical constraint 
\begin{align}
\label{NormalizedConstraint}
\dot\ell\at{t}=+1
\end{align}
\item 
and well-prepared initial data
\begin{align}
\label{WellPreparedData}
x_\ini\at{p} = \delta\,\bat{p-\xi_\ini} + \sgn\bat{p-
\xi_\ini}\,.
\end{align} 
\end{enumerate}
The special choice \eqref{WellPreparedData} implies for all sufficiently small times that no particle can penetrate the \mbox{spinodal} region and that the constant slope of $x$ persists. More precisely, by direct \mbox{computations} we verify for all times $0\leq t\leq \at{1-\ka}$ that the unique solution to the non-local dynamics \eqref{MicroDynamics}+\eqref{MicroMultiplier}+\eqref{DefTheta}+\eqref{DefNonl} is given by the formulas
\begin{align*}
x\pair{t}{p}=x_\ini\at{p}+t\,,\qquad 
\si\at{t}=\tau+t+\tfrac12\,\delta-\delta\,\xi_\ini\,,\qquad \xi_-\at{t}=\xi_+\at{t}=\xi_\ini
\end{align*}
and
\begin{align*}
\sgn_\ka\bat{x\pair{t}{p}}=\sgn\bat{p-\xi_\ini}\,,
\end{align*}
which describe a linear temporal grow of the state in presence of a standing interface. At time $t=1-\ka$, however,  the particle at $p=\xi_\ini$ penetrates the spinodal region from below and the phase interface starts moving to the left. Figure \ref{Fig:TWForm1} illustrates the first regime. After a transient time, the dynamical solution resembles a left moving traveling wave that propagates with constant speed and exhibits a constant interface width. See also Theorem~\ref{Thm:TW} and the left panel in Figure \ref{Fig:TW}. For other parameters, however, the width of the phase interface oscillates rather strongly and the solution cannot be approximated by a traveling wave. This is exemplified in Figures \ref{Fig:TWForm2} and \ref{Fig:TWForm3}.
\begin{figure}[ht!]%
\centering{%
\includegraphics[width=0.95\textwidth]{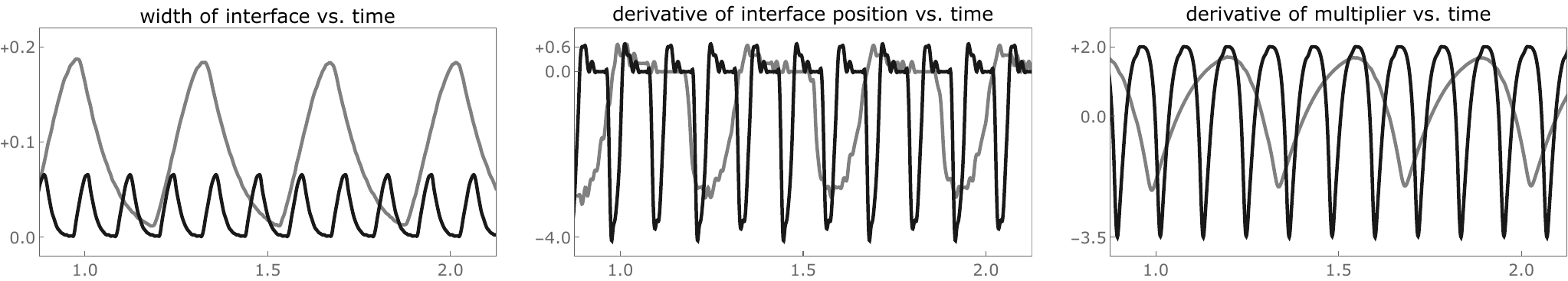}
}%
\caption{%
Oscillatory evolution of the interface width $\xi_+\at{t}-\xi_-\at{t}$, the interface speed $\tfrac12\dot{\xi}_+\at{t}+\tfrac12\dot{\xi}_-\at{t}$, and the change of the mean-field $\dot{\si}\at{t}$. The parameters are 
$\ka=0.5$, $\delta=0.5$ and $\tau=0.05$ (gray, as in Figure \ref{Fig:TWForm2}) or $\tau=0.01$ (black).
}
\label{Fig:TWForm3}%
\end{figure}%
%
%
%-------------------------------------------------------------------------------------------
\subsection{Explicit expressions for traveling waves} 
\label{sect:existenceTW}
%-------------------------------------------------------------------------------------------
%
Traveling wave solutions exist for all admissible parameters $\triple{\ka}{\delta}{\tau}$ but require a consistent \mbox{dynamical} constraint. We now derive the corresponding formulas by solving scalar ODEs with piecewise constant coefficients.
\begin{figure}[ht!]%
\centering{%
\includegraphics[width=0.85\textwidth]{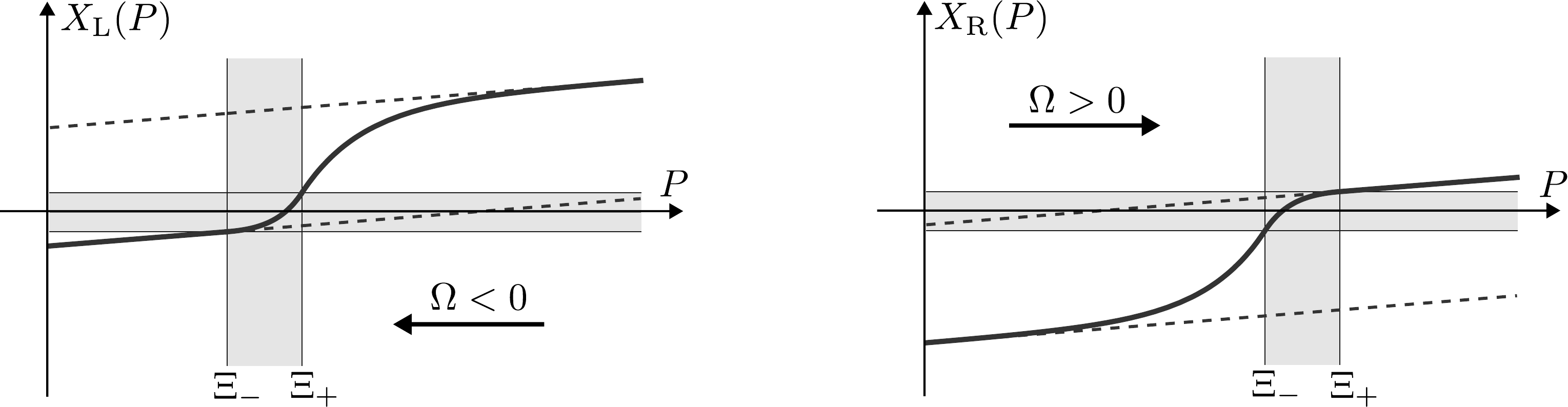}
}%
\caption{%
The two families of traveling waves (solid curves)  from Theorem \ref{Thm:TW} with negative or positive wave speed $\Omega$ along with the two affine envelopes (dashed lines) which have slope $\delta$ and differ by the constant $2$.  The vertical and horizontal gray stripes represent the spinodal region and the phase interface, respectively, where the width of the latter scales like $\tau\abs{\Om} \ln\at{1/\tau}$ according to \eqref{Lem:TW:Eqn4}. The exponential tail in the back of the wave describes the essential small scale effects in the bulk and has the characteristic width
$\tau\abs\Om$.}%
\label{Fig:TW}%
\end{figure}%
\begin{theorem}
\label{Thm:TW}
Equation \eqref{MicroDynamics} admits for the constitutive relations \eqref{DefTheta}$+$\eqref{DefNonl} and any admissible choice of $\tau$, $\delta$, $\ka$   two families of left and right moving traveling wave solutions
\begin{align}
\label{Lem:TW:Eqn0}
x_{\mathrm{TW\,}}\pair{t}{p}=X_{\mathrm{L}/\mathrm{R}}\at{P}\,,\qquad P=p-\Om\,t\,,\qquad \si_{\mathrm{TW\,}}\at{t} = \Sigma_{\mathrm{L}/\mathrm{R}}-\delta\,\at{\Om\,t-\tfrac12}
\end{align}
which are strictly increasing with respect to $p$ and illustrated in Figure \ref{Fig:TW}. These waves are given by
\begin{align}
\notag%\label{Lem:TW:Eqn1}
\begin{split}
X_{\mathrm{L}}\at{P}&=\left\{
\begin{array}{lr}
\!\!\!-\ka+\delta\,\at{P-\Xi_-}&\!\!\!\!\!\!\!\!\!\text{if}\;\;P{\,\leq\,}\Xi_-\\
\!\!\!-\ka-\D\frac{\ka\,\delta}{1-\ka}\,\at{P-\Xi_-}-\frac{\ka\,\tau\,\Om\,\delta}{\at{1-\ka}^2}\,\at{
\exp\at{-\frac{1-\ka}{\ka\,\tau\,\Om}\,\at{P-\Xi_-}}-1
}&\!\!\!\!\!\!\!\!\!\text{if}\;\;\Xi_-{\,\leq\,}P{\,\leq\,}\Xi_+\\
\!\!\!+\ka+\delta\,\at{P-\Xi_+}+\Bat{2\,\at{1-\ka}+\delta\,\at{\Xi_+-\Xi_-}}\at{+1-\exp\at{\D\frac{P-\Xi_+}{\tau\,\Om}}}&\!\!\!\!\!\!\!\!\!\text{if}\;\;\Xi_+{\,\leq\,}P\\
\end{array}\right.
\\
\Sigma_{\mathrm{L}} &= +1-\ka-\tau\,\Om\,\delta -\delta\, \Xi_-
\end{split}
\end{align}
and
\begin{align}
\notag%\label{Lem:TW:Eqn2}
\begin{split}
X_{\mathrm{R}}\at{P}&=
\left\{
\begin{array}{lr}
\!\!\!-\ka+\delta\,\at{P-\Xi_-}+\Bat{2\,\at{1-\ka}+\delta\,\at{\Xi_+-\Xi_-}}\at{-1+\exp\at{\D\frac{P-\Xi_-}{\tau\,\Om}}}&\!\!\!\!\!\!\!\!\!\text{if}\;\;P{\,\leq\,}\Xi_-\\
\!\!\!+\ka-\D\frac{\ka\,\delta}{1-\ka}\,\at{P-\Xi_+}-\frac{\ka\,\tau\,\Om\,\delta}{\at{1-\ka}^2}\,\at{
\exp\at{-\frac{1-\ka}{\ka\,\tau\,\Om}\,\at{P-\Xi_+}}-1
}&\!\!\!\!\!\!\!\!\!\text{if}\;\;\Xi_-{\,\leq\,}P{\,\leq\,}\Xi_+\\
\!\!\!+\ka+\delta\,\at{P-\Xi_+}&\!\!\!\!\!\!\!\!\!\text{if}\;\;\Xi_+{\,\leq\,}P
\end{array}\right.\\
\Sigma_{\mathrm{R}} &= -1+\ka-\tau\,\Om\,\delta -\delta\, \Xi_+
\end{split}
\end{align}
for $\Om<0$ and $\Om>0$, respectively, provided that the transcendental equation
 \begin{align}
\label{Lem:TW:Eqn3}
\frac{2\,\at{1-\ka}^2}{\tau\abs{\Om}\delta}+\frac{1-\ka}{\tau\abs{\Om}}\,\bat{\Xi_+-\Xi_-}
=
\exp\at{\D\frac{1-\ka}{\ka\,\tau\abs{\Om}}\,\bat{\Xi_+-\Xi_-}}-1\,,
\end{align}
is satisfied. In particular, the wave speed $\Om$ can be regarded as the independent parameter and determines the interface width $\Xi_+-\Xi_-$ but not the interface position $\tfrac12\at{\Xi_-+\Xi_+}$.
\end{theorem}
\begin{proof}
We discuss the case $\Om<0$ only and do not write a lower index in $X$ or $\Sigma$. The ansatz \eqref{Lem:TW:Eqn0} combined with 
\begin{align*}
X\at{P}<-\ka\;\;\text{for}\;\;
P<\Xi_-\,,\quad
-\ka<X\at{P}<+\ka\;\;\text{for}\;\;
\Xi_-<P<\Xi_+\,,\quad
X\at{P}>+\ka\;\;\text{for}\;\;P>\Xi_+
\end{align*}
and \eqref{DefNonl} transforms the dynamical model \eqref{MicroDynamics} into a scalar first order ODE with piecewise constant coefficients, namely
\begin{align}
\label{Lem:TW:PEqn1}
-\tau\,\Om\,X^\prime\at{P} - \delta\, P -\Sigma =\left\{
\begin{array}{cl}
- X\at{P}-1&\text{for $P<\Xi_-$\,,}\smallskip\\
+ \D \frac{1-\ka}{\ka}X\at{P}&\text{for $\Xi_-<P<\Xi_+$\,,}
\smallskip\\
- X\at{P}+1&\text{for $P>\Xi_+$}\,,
\end{array}\right.
\end{align}
where $P=p-\Om\,t$ denotes the variable in the comoving frame. The only non-exponential solution in front of the interface is
\begin{align*}
X\at{P}=\delta\,P+\Sigma+\tau\,\Om\,\delta-1\qquad\text{for}\quad P<\Xi_-
\end{align*}
and the first matching condition 
\begin{align}
\label{Lem:TW:PEqn2}
X\at{\Xi_-}=-\ka  
\end{align}
yields the formula for $\Sigma$. In the interface region, the ODE in \eqref{Lem:TW:PEqn1} implies 
\begin{align*}
X\at{P}=-\frac{\ka}{1-\ka}\,\delta\,\at{P-\Xi_-}-\ka+\frac{\ka\,\tau\,\Om\,\delta}{\at{1-\ka}^2}+C_0\exp\at{-\frac{1-\ka}{\ka\,\tau\,\Om}\,\at{P-\Xi_-}}
\qquad\text{for}\quad \Xi_-<P<\Xi_+
\end{align*}
and \eqref{Lem:TW:PEqn2} ensures that the constant of integration attains the value $C_0=-\ka\,\tau\,\Om\,\delta/\at{1-\ka}^2$. Moreover, the second matching condition
\begin{align}
\label{Lem:TW:PEqn3}
X\at{\Xi_+}=+\ka  
\end{align}
is equivalent to  \eqref{Lem:TW:Eqn3}. Finally, from \eqref{Lem:TW:Eqn4} we infer 
\begin{align*}
X\at{P}=\delta\,P+2-\ka-\delta\,\Xi_-+C_+\exp\at{\frac{P-\Xi_+}{\tau\,\Om}\,}\qquad\text{for}\quad P>\Xi_+
\end{align*}
and $C_+=-2\,\at{1-\ka}-\delta\,\at{\Xi_+-\Xi_-}$ is an immediate consequence of  \eqref{Lem:TW:PEqn3}.
\end{proof}
\paragraph{Remarks}
\begin{enumerate}
\item 
The left hand side in condition \eqref{Lem:TW:Eqn3} is linear and increasing with respect to $\Xi_+-\Xi_-$, while the right hand side is also increasing but strictly convex. Comparing the respective values at $0$ and $\infty$ we conclude that for any $\Om\neq 0$ there exists a unique positive interface width. Moreover, by asymptotic standard techniques we show
\begin{align}
\label{Lem:TW:Eqn4}
\Xi_+-\Xi_-=\frac{\tau\abs{\Om}\ka}{1-\ka}\,\ln\at{\frac{2\,\at{1-\ka}^2}{\tau\abs{\Om}\delta}\,\Bat{1+O\at{\tau\abs{\Om}}}}
=\frac{\abs{\Om}\ka}{1-\ka}\,\tau\,\ln\at{1/\tau}\,\Bat{1+o\at{1}}\,,
\end{align} 
where $o\at{1}$ means arbitrarily small for small $\tau$. Notice that the left hand side of \eqref{Lem:TW:Eqn3} and the middle part in \eqref{Lem:TW:Eqn4} are not defined for $\delta=0$, which is another indication for the importance of the \BMHC inhomogeneities modeled by $\theta$.\EMHC
\item 
The traveling wave solutions from Theorem \ref{Thm:TW} are formally defined for all $p\in\Rset$ and do not involve any non-local side condition. However, the formulas can also be evaluated on the interval $p\in\ccinterval{0}{1}$ and provide via
\begin{align}
\label{TWConstraint}
\ell_{\text{TW}}\at{t}=\int\limits_{\Xi_--\Om\,t}^{\Xi_+-\Om\,t} X_{\mathrm{L}/\mathrm{R}}\at{P}\dint{P}=\at{2-\ka+\tfrac12\,\delta}-\at{2+\delta}\bat{\Xi-\Om\,t}+O\at{\tau\,\ln\at{\frac{1}{\tau}}}
\end{align}
the value of the corresponding time dependent dynamical constraint, where $\Xi=\tfrac12\at{\Xi_-+\Xi_+}$. The error terms stem from the small interface width as well as the exponential tails in the back, but the main part grows linearly in time and is consistent with the limit model as described in \S\ref{sect:Intro}.
\item 
In the bilinear limiting case $\ka=0$, we have $\Xi_-=\Xi_+=\Xi$ and the traveling wave formulas reduce to
\begin{align*}
X_{\mathrm{L}}\at{P}&=\delta\,\at{P-\Xi}+\left\{
\begin{array}{lcr}
0&\text{for}&P<\Xi\,,\\2\,\at{+1-\exp\at{\D\frac{P-\Xi}{\tau\,\Om}}}&\text{for}&P>\Xi\,,\\
\end{array}\right.
\end{align*}
and
\begin{align*}
X_{\mathrm{R}}\at{P}&=
\delta\,\at{P-\Xi}+
\left\{
\begin{array}{lcr}
2\,\at{-1+\exp\at{\D\frac{P-\Xi}{\tau\,\Om}}}&\text{for}&P<\Xi\,,\\
0&\text{for}&P>\Xi\,.
\end{array}\right.\,.
\end{align*}
In particular, the interface width vanishes in the limit $\ka\to0$ while the exponential tails in the back are still present. \BMHC Similar formulas have been derived in \cite[Section 3.2]{TV08} for traveling wave solutions  to viscoelastic wave equations.\EMHC 
\end{enumerate}
\BMHC Of course, Theorem \ref{Thm:TW} is intimately related to 
\eqref{DefTheta} and \eqref{DefNonl}, the piecewise \mbox{constitutive} assumptions for $\Phi^\prime$ and $\theta$. However,
numerical solutions indicate that stable traveling waves solutions are also relevant in a more general setting and still describe, at least in some parameter regimes, the fine structure of moving phase interfaces. The key idea is that both the width and the speed are no longer constant but adapt to
the respective values in a traveling wave that is compatible with the current values of $\dot\ell\at{t}$ and $\delta\at{t}=\theta^\prime\at{\xi_\pm\at{t}}$. A similar relaxation process has been investigated in \cite{TV10b} for phase transition waves in a lattice of overdamped viscoelatic springs with external forcing. That paper also proposes a two-dimensional gradient ODE for the effective propagation of the phase interface which accounts for oscillations on small time scales. For the mean-field model \eqref{MicroDynamics}+\eqref{MicroMultiplier} it remains a challenging task to derive similar low-dimensional dynamical equations that capture the temporal oscillations in the regime of small imhomogeneities as displayed in Figures \ref{Fig:TWForm2} and \ref{Fig:TWForm3}.
\EMHC 
%
%
%-------------------------------------------------------------------------------------------
\subsection{Linearized equation in the comoving frame}
%-------------------------------------------------------------------------------------------
%
%
Linearizing \eqref{MicroDynamics}+\eqref{MicroMultiplier} around a traveling wave solution  provided by Theorem \ref{Thm:TW} we obtain the linear but \mbox{non-autonomous} equation
\begin{align}
\label{LinStab.Dynamics}
\tau\,\partial_t z\pair{t}{p} = -z\pair{t}{p}+\ka^{-1}\,\psi\pair{t}{p}\,z\pair{t}{p}-\ka^{-1}\int\limits_{0}^{1}\psi\pair{t}{q}\, z\pair{t}{q}\dint{q}\,,
\end{align}
where the coefficient function
\begin{align*}
\psi\pair{t}{p}=\left\{
\begin{array}{ccl}
1&&\text{for}\;\;\;
\Xi_-<p-\Om\,t<\Xi_+\\0&&\text{otherwise}
\end{array}
\right.
\end{align*}
represents the stripe in which $x_{\mathrm{TW}}$ attains values inside the spinodal region. \BMHC See Figure \ref{Fig:LinCoeff} for an illustration and notice that the Cauchy Problem to \eqref{LinStab.Dynamics} is well-posed in any $\fspaceL^p$-space. \EMHC The integral term in  \eqref{LinStab.Dynamics} stems from the multiplier rule \eqref{MicroMultiplier} and ensures via
\begin{align*}
\tau\,\dot{m}\at{t} =-m\at{t}\qquad\text{with}\qquad  m\at{t}=\int\limits_0^1z\pair{t}{p}\dint{p}
\end{align*}
the analogue to the dynamical constraint \eqref{MicroConstraint}. In particular, $m\at{0}=0$ implies $m\at{t}=0$ for all $t\geq0$.
\begin{figure}[ht!]%
\centering{%
\includegraphics[width=0.35\textwidth]{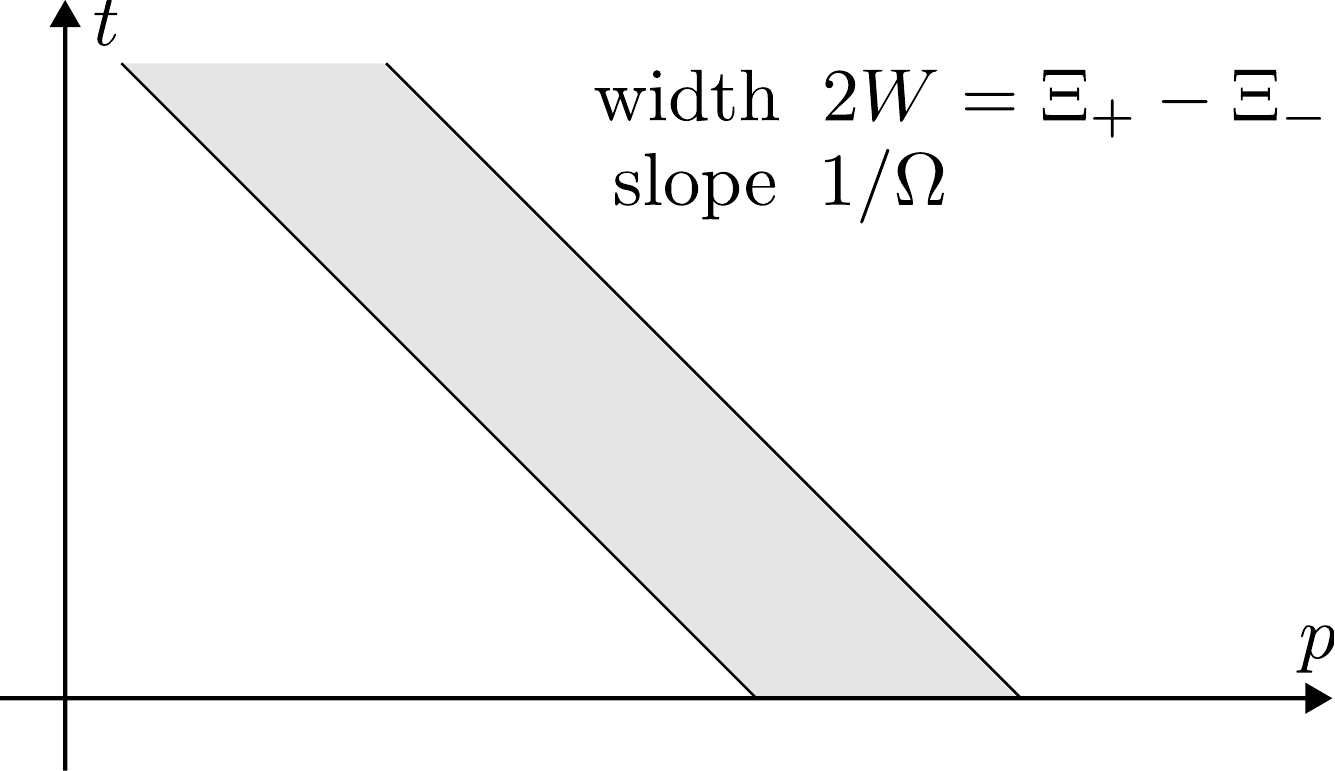}
}%
\caption{%
The support of the piecewise constant coefficient function $\psi$ in the linearized model \eqref{LinStab.Dynamics}. For the traveling wave from Theorem \ref{Thm:TW}, the parameters $\ka$, $\tau$, $\Om$ and $W$ are coupled by the scalar equation \eqref{Lem:TW:Eqn3} but the evolution equation \eqref{LinStab.Dynamics} can also be studied without that condition. On a heuristic level we then expect that the trivial solution switches from stable to unstable if $W$ exceeds a certain critical value (depending on $\ka$ and $\Om$) or if $\Omega$ falls below a certain threshold (depending on $\ka$ and $W$). 
}%
\label{Fig:LinCoeff}%
\end{figure}%
\bigpar
For the stability analysis it is convenient to regard $p$ as a variable in $\Rset$ and to pass to the co-moving frame so that the coefficient function $\psi$ 
transforms into the stationary indicator function
\begin{align*}
\Psi\at{P}=
\left\{
\begin{array}{ccl}
1&&\text{if}\;\; 
\abs{P}<W,\\0&&\text{else},
\end{array}
\right.
\end{align*}
where 
\begin{align}
\label{DefWidth}
W=\tfrac12\at{\Xi_+-\Xi_-}
\end{align}
denotes the half of the interface width in the underlying traveling wave. The eigenvalue problem of the transformed equation reads
\begin{align}
\label{LinStab.EP}
 \tau\, \la\,Z\at{P} = \tau\, \Om\,Z^\prime\at{P}-Z\at{P}+\ka^{-1}\,\Psi\at{P}\,Z\at{P}-\ka^{-1}\int\limits_{-\infty}^{+\infty} \Psi\at{Q}\,Z\at{Q}\dint{Q}\,,\qquad 
\end{align} 
and describes after the transformation
\begin{align}
\label{EVTrafo}
z\pair{t}{p}=\exp\at{\la\,t}\,Z\at{P}\,,\qquad P=p-\tfrac12\at{\Xi_-+\Xi_+}-\Om\,t
\end{align}
the fundamental modes of the linear but non-autonomous equation \eqref{LinStab.Dynamics} that corresponds to the eigenvalue $\la\in\Cset$.
\par
The key idea for our analysis is to regard \eqref{LinStab.EP} as three ODEs with constant coefficients that are coupled by two matching conditions and one integral constraint.  More precisely, setting
\begin{align*}
Z\at{P}=\left\{\begin{array}{cclcl}%
Z_-\at{P}&&\text{for }P\in I_-&\!\!\!\!:=\!\!\!\!&\oointerval{-\infty}{-W}\\
Z_0\at{P}&&\text{for }P\in I_0&\!\!\!\!:=\!\!\!\!&\oointerval{-W}{+W}\\
Z_+\at{P}&&\text{for }P\in I_+&\!\!\!\!:=\!\!\!\!&\oointerval{+W}{+\infty}
\end{array}\right.%
\end{align*}
we readily verify  the differential equations
\begin{align}
\label{LinStab.ODEs}
\at{\ka\,\tau\,\la +\ka}\,Z_\mp\at{P} = \ka\,\tau\,\Om\, Z^\prime_\mp\at{P}-\zeta\,, \qquad\quad \at{\ka\,\tau\,\la +\ka-1}\,Z_0\at{P} = \ka\,\tau\,\Om\, Z^\prime_0\at{P}-\zeta\,,
\end{align}
where the non-local mean-field  is given by
\begin{align}
\label{LinStab.Nonlocal}
\zeta = \int\limits_{-W}^{+W} Z_0\at{q}\dint{q}\,.
\end{align}
Moreover, the equations
\begin{align}
\label{LinStab.Matching}
Z_-\at{-W}=Z_0\at{-W}\,,\qquad Z_0\at{+W}=Z_+\at{+W} 
\end{align}
hold in the sense of one-sided limits and guarantee the continuity of $Z$ at $P=\mp W$.
%
%
%-------------------------------------------------------------------------------------------
\subsection{Characterization of the spectrum}
%-------------------------------------------------------------------------------------------
%
%
We first relate the complex eigenvalue to a transcendental equation that involves both exponential and polynomial terms in $\la$. This result regards $\ka$, $\tau$, $\Om$ and $W$ as independent parameters and is independent of \eqref{Lem:TW:Eqn3}. Numerical results are presented in Figure \ref{Fig:SpectrumG}.
\begin{figure}[ht!]%
\centering{%
\includegraphics[width=0.95\textwidth]{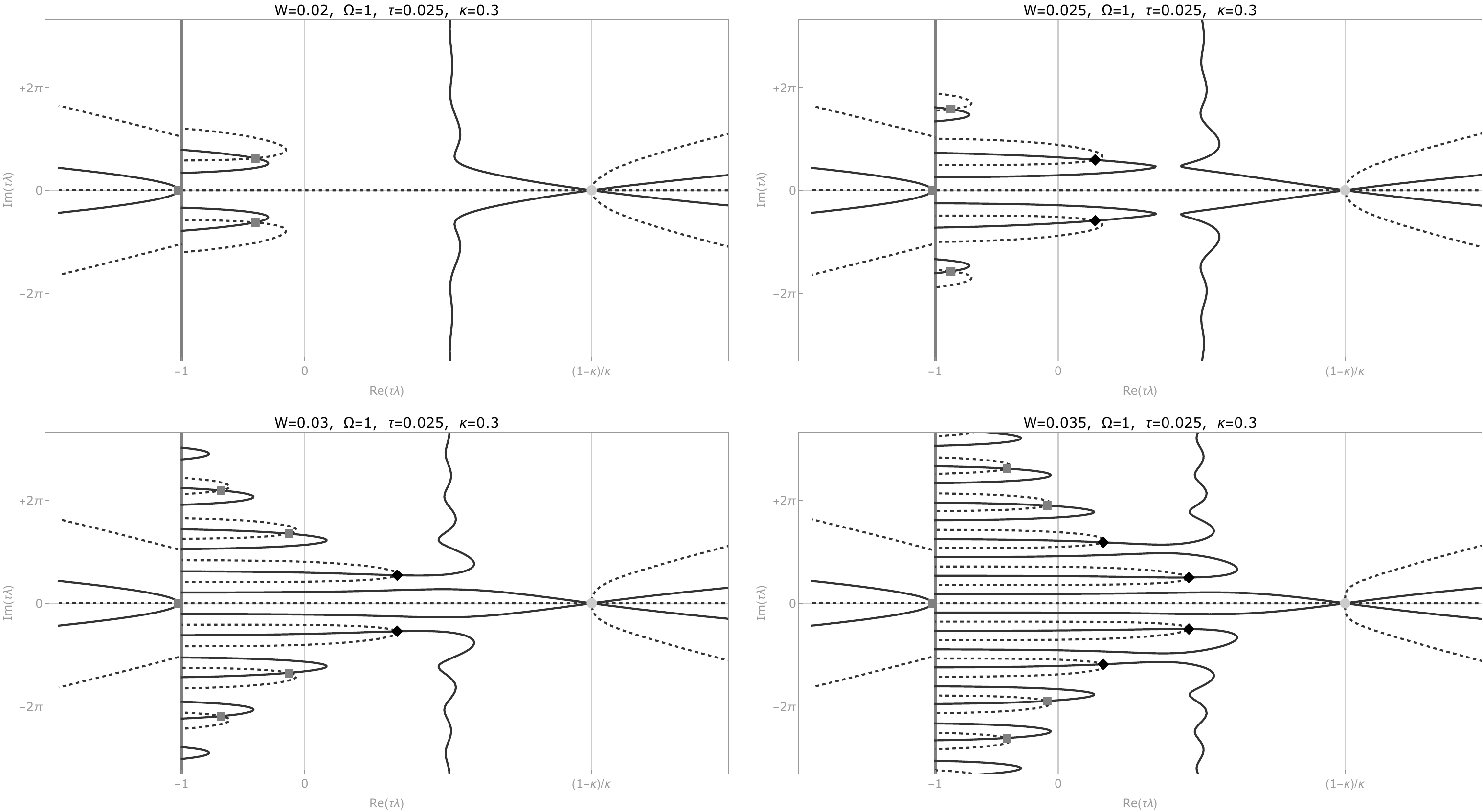}
}%
\caption{%
The spectrum of the non-local linear operator in \eqref{LinStab.EP} for fixed values of $\tau$, $\ka$, $\Om$ and three different choices for $W$. Stable and unstable eigenvalues correspond to gray squares and black diamonds, respectively, while the continuous part coincides with the gray vertical line. The solid and dashed curves represent the zero sets of the real and the imaginary parts of \eqref{LinStab.Thm1}+\eqref{LinStab.Thm2}, respectively, but the light gray point at $\at{1-\ka}/\ka$ does not belong to the  spectrum.
}%
\label{Fig:SpectrumG}%
\end{figure}%
\begin{theorem}\label{Thm:Spectrum}
The linear eigenvalue problem \eqref{LinStab.EP} has the following spectral properties in the space of all functions that grow at most  linearly.
\begin{enumerate}
\item[] \emph{\ul{point spectrum, part $S_-$\,:}} 
The complex number $\la$ with $\mhRe\at{\tau\,\la}<-1$ is an eigenvalue if and only if it satisfies
\begin{align}
\label{LinStab.Thm1}
\exp\at{+\frac{\ka\,\tau\,\la+\ka-1}{\ka\,\tau\abs{\Om}}\,2\,W}-1=\frac{\;\bat{\ka\,\tau\,\la+\ka}\,\bat{\ka\,\tau\,\la+\ka-1}\,\bat{\ka\,\tau\,\la+\ka-1+2\,W}\;}{\ka\,\tau\abs{\Om}}\,.
\end{align}
\item[] \emph{\ul{continuous spectrum $S_0$\,:}}
For every $\la\in\Cset$ with $\mhRe\at{\tau\,\la}=-1$ there exists a corresponding eigenfunction.
\item[] \emph{\ul{point spectrum, part $S_+$\,:}} 
The complex number $\la$ with $\mhRe\at{\tau\,\la}>-1$ is an eigenvalue if and only if it satisfies
\begin{align}
\label{LinStab.Thm2}
1-\exp\at{-\frac{\ka\,\tau\,\la+\ka-1}{\ka\,\tau\abs{\Om}}\,2\,W}=\frac{\;\bat{\ka\,\tau\,\la+\ka}\,\bat{\ka\,\tau\,\la+\ka-1}\,\bat{\ka\,\tau\,\la+\ka-1+2\,W}\;}{\ka\,\tau\abs{\Om}}
\end{align}
as well as $\tau\,\la\neq \at{1-\ka}/\ka$.
\end{enumerate} 
Moreover, all proper eigenvalues  $\la\in S_-\cup S_+$ are simple and have bounded eigenfunctions.
\end{theorem}
\begin{proof}
We present the arguments for $\Om>0$ only. The assertions for $\Om<0$ can be derived along the sames lines. We further start with the assumption
\begin{math}
\ka\,\tau\,\la+\ka\notin\{0,\,1\} 
\end{math} 
and discuss the two remaining special cases, in which the ODE system \eqref{LinStab.ODEs} degenerates, at the very end of this proof.
\par
\emph{\ul{Non-vanishing mean-field}}\,:  
Assuming $\zeta=0$ we deduce from the ODEs \eqref{LinStab.ODEs} that $Z_-$, $Z_0$, and $Z_+$  are all exponential functions in the variable $P$. The non-local equation \eqref{LinStab.Nonlocal} then implies $Z_0\equiv0$ and $Z_-\equiv 0\equiv Z_+$ follows from the matching conditions in \eqref{LinStab.Matching}. We have thus shown that any nontrivial eigenfunction corresponds to $\zeta\neq0$
\par
\emph{\ul{Solution formulas}}\,: 
Applying standard arguments to each ODE in \eqref{LinStab.ODEs} we get
\begin{align}
\label{LinStab.ThmProof11}
Z_\mp\at{P}=-\frac{\zeta}{\ka\,\tau\,\la+\ka}+\zeta\,C_{\mp}\exp\at{\frac{\ka\,\tau\,\la+\ka}{\ka\,\tau\,\Om}\,P}\qquad\text{for $P\in I_{\mp}$}
\end{align}
as well as
\begin{align}
\label{LinStab.ThmProof12}
Z_0\at{P}=-\frac{\zeta}{\ka\,\tau\,\la+\ka-1}+\zeta\,C_{0}\exp\at{\frac{\ka\,\tau\,\la+\ka-1}{\ka\,\tau\,\Om}\,P}\qquad\text{for $P\in I_{0}$}\,,
\end{align}
where $C_-$, $C_0$ and $C_+$ represent the constants of integration. Moreover, direct computations reveal that 
\begin{align}
\label{LinStab.ThmProof13}
C_0=\frac{\D \frac{\ka\,\tau\,\la+\ka-1 + 2\,W}{\ka\,\tau\,\Om}}{\;\;\D\exp\at{+\frac{\ka\,\tau\,\la+\ka-1}{\ka\,\tau\,\Om}\,W}-\exp\at{-\frac{\ka\,\tau\,\la+\ka-1}{\ka\,\tau\,\Om}\,W}\;\;}
\end{align}
is equivalent to the constraint in \eqref{LinStab.Nonlocal}.
\par
\emph{\ul{Case $\mhRe\at{\tau\,\la}<-1$}}\,: For any eigenfunction we have 
\begin{align}
\label{LinStab.ThmProof21}
C_-=0
\end{align}
because otherwise $Z_-$ would grow exponentially for $P\to-\infty$. The matching condition at $P=-W$ thus implies
\begin{align}
\label{LinStab.ThmProof22}
C_0=\frac{1}{\;\at{\ka\,\tau\,\la+\ka}\,\at{\ka\,\tau\,\la+\ka-1}}\,\exp\at{+\frac{\ka\,\tau\,\la+\ka-1}{\ka\,\tau\,\Om}\,W}
\end{align}
thanks to \eqref{LinStab.ThmProof11} and \eqref{LinStab.ThmProof12} and in combination with \eqref{LinStab.ThmProof13} we identify \eqref{LinStab.Thm1} as a necessary condition after elementary computations.  On the other hand, the validity of \eqref{LinStab.Thm1} ensures that \eqref{LinStab.ThmProof11} and \eqref{LinStab.ThmProof12} yield for any $\zeta\neq0$ a bounded eigenfunction provided that the constants are chosen by \eqref{LinStab.ThmProof21}, \eqref{LinStab.ThmProof22} and 
\begin{align*}
C_+=\frac{1}{\;\at{\ka\,\tau\,\la+\ka}\,\at{\ka\,\tau\,\la+\ka-1}}\at{\exp\at{+\frac{\ka\,\la+\ka-2}{\ka\,\tau\,\Om}\,W}-\exp\at{-\frac{\ka\,\tau\,\la+\ka}{\tau\,\ka\,\Om}\,W}}\,,
\end{align*}
where the latter identity reflects  the matching condition at $P=+W$.
\par
\emph{\ul{Case $\mhRe\at{\tau\,\la}>-1$}}\,: 
The growth restriction for $Z$ requires
\begin{align*}
C_+=0
\end{align*} 
and the matching condition at $P=+W$ implies
\begin{align*}
C_0=\frac{1}{\;\at{\ka\,\tau\,\la+\ka}\,\at{\ka\,\tau\,\la+\ka-1}}\, \exp\at{-\frac{\ka\,\tau\,\la+\ka-1}{\ka\,\tau\,\Om}\,W}\,,
\end{align*}
which is due to \eqref{LinStab.ThmProof13} equivalent to \eqref{LinStab.Thm2}. Moreover, the choice
\begin{align*}
C_-=\frac{1}{\;\at{\ka\,\tau\,\la+\ka}\,\at{\ka\,\tau\,\la+\ka-1}}\, \at{\exp\at{-\frac{\ka\,\tau\,\la+\ka-2}{\ka\,\tau\,\Om}\,W}-
\exp\at{+\frac{\ka\,\tau\,\la+\ka}{\tau\,\ka\,\Om}\,W}}
\end{align*}
guarantees the matching condition at $P=-W$.
\par
\emph{\ul{Case $\mhRe\at{\tau\,\la}=-1$}}\,: 
Due to $\ka\,\tau\,\la+\ka=\iu\,\nu$ neither $C_-$ or $C_+$ must be zero. The solution formulas \eqref{LinStab.ThmProof11} and \eqref{LinStab.ThmProof12} thus provide for any choice of $\nu\neq0$ and $\zeta \neq0$ a nontrivial but bounded eigenfunction provided that we first compute $C_0$ as in \eqref{LinStab.ThmProof13} and choose afterwards both $C_-$ and $C_+$ to satisfy the matching conditions at $P=-W$ and $P=+W$.
\par
\emph{\ul{Special case $\tau\,\la=-1$}}\,: 
On the interval $I_0$ we argue as in the previous case and obtain 
\begin{align*}
Z_0\at{P}=+\zeta+\zeta\,C_0\,\exp\at{-\frac{P}{\ka\,\tau\,\Om}}\,,\qquad 
C_0=\frac{\D \frac{2\,W-1}{\ka\,\tau\,\Om}}{\;\;\D\exp\at{-\frac{W}{\ka\,\tau\,\Om}}-\exp\at{+\frac{W}{\ka\,\tau\,\Om}}\;\;}
\end{align*}
by evaluating \eqref{LinStab.ThmProof12} and \eqref{LinStab.ThmProof13}, where $\zeta$ is again a free parameter. From the degenerate ODEs on $I_-$ and $I_+$ (see \eqref{LinStab.ODEs}) we then deduce
\begin{align*}
Z_-\at{P}=\frac{\zeta}{\ka\,\tau\,\Om}\,P+\zeta\,D_-\,,\qquad Z_+\at{P}=\frac{\zeta}{\ka\,\tau\,\Om}\,P+\zeta\,D_+\,,
\end{align*}
where 
\begin{align*}
D_-=1+\frac{W}{\ka\,\tau\,\Om}+C_0\exp\at{+\frac{W}{\ka\,\tau\,\Om}}\,,\qquad 
D_+=1-\frac{W}{\ka\,\tau\,\Om}+C_0\exp\at{-\frac{W}{\ka\,\tau\,\Om}}
\end{align*}
is a consequence of the matching conditions \eqref{LinStab.Matching}.
\par
\emph{\ul{Special case $\tau\,\la=\at{1-\ka}/\ka$}}\,: 
The degenerate ODE for $Z_0$ in \eqref{LinStab.ODEs} implies
\begin{align*}
Z_0\at{P}=\frac{\zeta}{2\,W}+\frac{\zeta}{\ka\,\tau\,\Om}\,P\qquad\text{and hence}\qquad Z_0\at{+W}=\frac{\zeta}{2\,W}+\frac{\zeta\,W}{\ka\,\tau\,\Om}\,,
\end{align*}
where the constant of integration has been determined by \eqref{LinStab.Nonlocal}. On the other hand, since $C_+$ must vanish we get $Z_+\at{W}=-\zeta$ and the matching condition at $P=+W$  ensures $\zeta =0$.  We thus conclude that this particular value of $\la$ does not belong to the spectrum, although it satisfies equation \eqref{LinStab.Thm2}.
\end{proof}
By elementary real analysis we show that $S_-$ contains for all sufficiently small $W$ precisely  one real eigenvalue and it seems that further elements do not exists. However, the properties of $S_-$ are not relevant in our context since the stability of the trivial solution to \eqref{LinStab.Dynamics} is completely determined by $S_+$.
%
%
%-------------------------------------------------------------------------------------------
\subsection{Instability of traveling waves}
%-------------------------------------------------------------------------------------------
%
The combination of the interface condition \eqref{Lem:TW:Eqn4} and the spectral equation \eqref{LinStab.Thm2} for $S_+$ enables us to study the stability of traveling waves by means of two dimensional plots as illustrated in Figures \ref{Fig:SpectrumTW1} and \ref{Fig:SpectrumTW2}. In the limit of vanishing $\tau$ we can even derive an explicit instability criterion from a suitable rescaling of the spectral equation for  $S_+$
\begin{lemma}\label{Lem:Unstable}
Using \eqref{Lem:TW:Eqn3} as well as
\begin{align}
\label{Cor:AsympSpectralScaling}
\mu= \frac{\tau\,\la}{\eps},\qquad \eps=\frac{\tau\abs{\Om}}{2\,W}
\end{align}
the transcendental equation \eqref{LinStab.Thm2} reads
\begin{align*}
\exp\at{+\mu}=-\frac{\ka}{\delta}\,\frac
{2\,\at{1-\ka}^2+\at{1-\ka}\,\delta\,2\,W+\tau\abs{\Om}\delta}
{\at{\ka\,\eps\,\mu+\ka}\,\at{\ka\,\eps\,\mu+\ka-1}\,\at{\ka\,\eps\,\mu+\ka-1+2\,W}-\ka\,\tau\abs{\Om}}\,.
\end{align*}
Moreover, the right hand side converges for $\tau\to0$ and pointwise in $\mu$ to $-2/\delta$.
\end{lemma}
\begin{figure}[ht!]%
\centering{%
\includegraphics[width=0.95\textwidth]{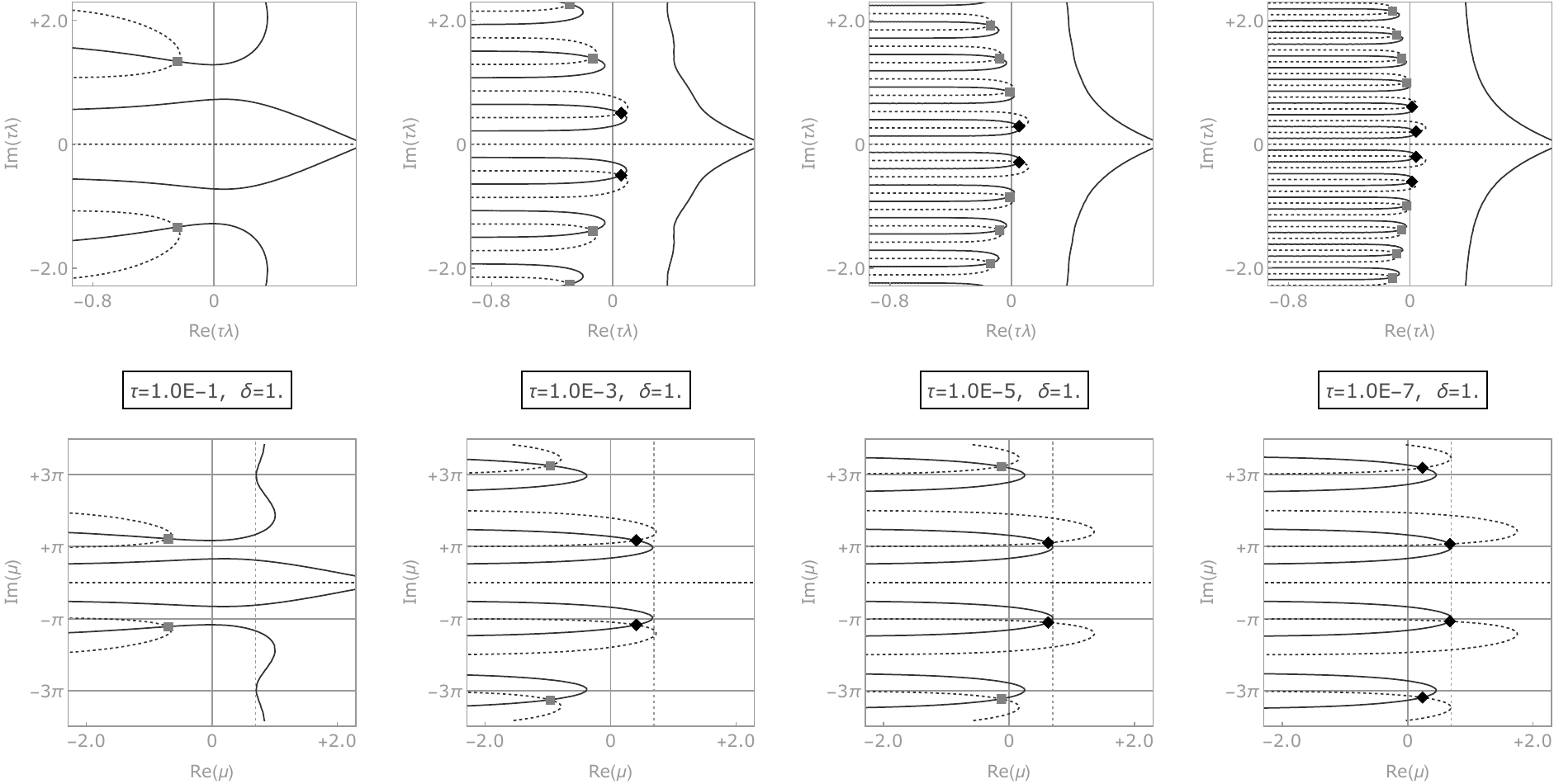}
}%
\caption{%
Asymptotic instability of the traveling waves for $\ka=0.5$, $\delta=1.0$, and $\abs{\Om}=1$. \emph{Top}. The solution set to  \eqref{LinStab.Thm2}
is shown for several choices of $\tau$ and similarly to Figure \ref{Fig:SpectrumG}. \emph{Bottom}. The eigenvalues near the complex origin are unstable for sufficiently small $\tau$ since the reals parts converge after rescaling as in \eqref{Cor:AsympSpectralScaling} to the limit $\ln\at{2/\delta}>0$ (dotted vertical line).
}%
\label{Fig:SpectrumTW1}%
\end{figure}%
\begin{figure}[ht!]%
\centering{%
\includegraphics[width=0.95\textwidth]{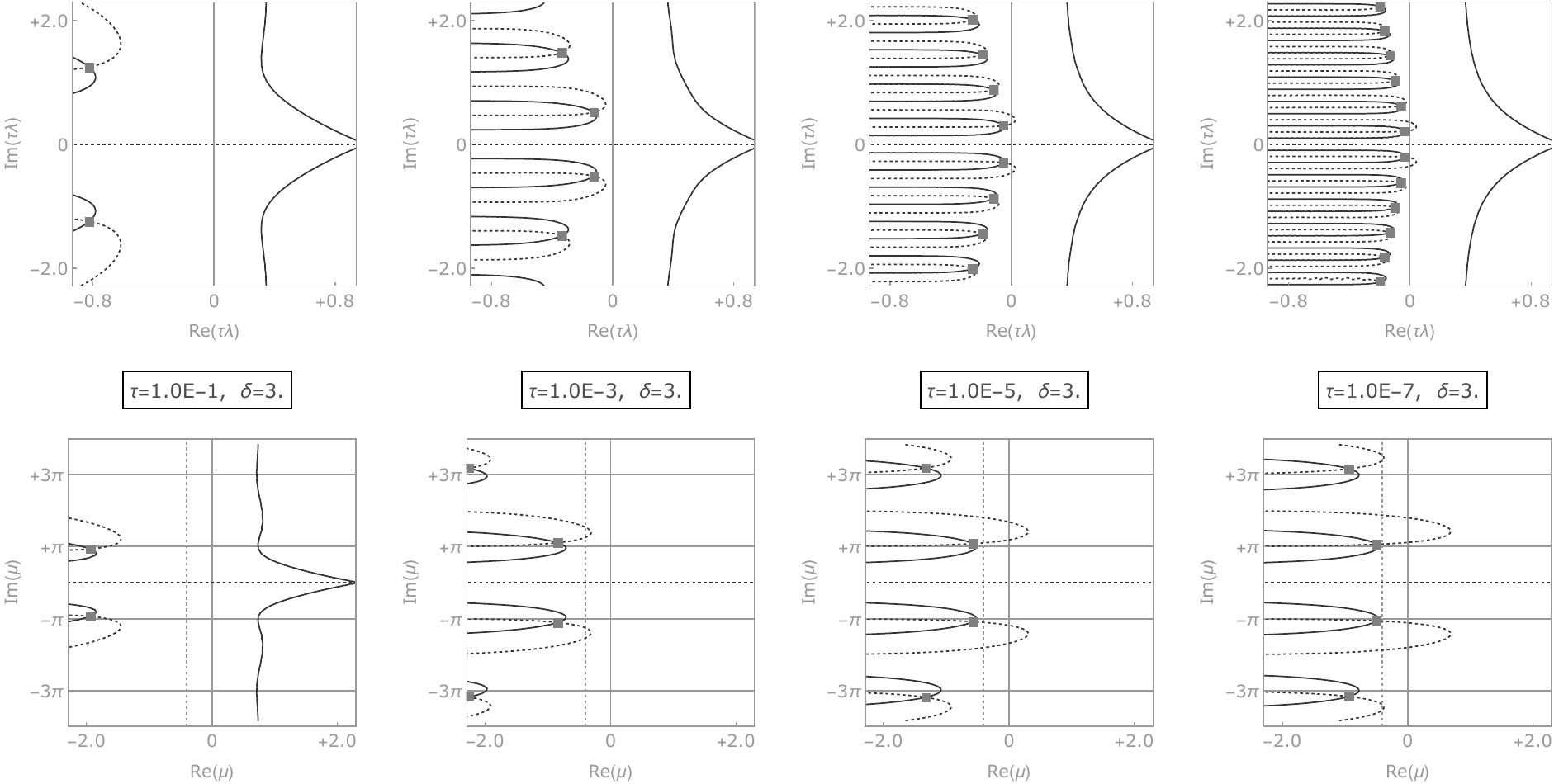}
}%
\caption{%
Recomputation of Figure \ref{Fig:SpectrumTW1} with modified value $\delta=3.0$. The eigenvalues near the complex origin now remain stable for $\tau\to0$ due to $\ln\at{2/\delta}<0$.
}%
\label{Fig:SpectrumTW2}%
\end{figure}%
\begin{proof}
From \eqref{Lem:TW:Eqn3} we deduce
\begin{align*}
\exp\at{\frac{1-\ka}{\ka\,\tau\abs\Om}\,2\,W}=\frac{2\,\at{1-\ka}^2+\at{1-\ka}\,\delta\,2\,W+\tau\abs{\Om}\delta}{\tau\abs{\Om}\delta}
\end{align*}
and \eqref{LinStab.Thm2} can  be written as
\begin{align*}
\exp\at{\frac{1-\ka}{\ka\,\tau\abs\Om}\,2\,W}
\,\exp\at{-\mu}=
-\frac{\at{\ka\,\tau\,\la+\ka}\,\at{\ka\,\tau\,\la+\ka-1}\,\at{\ka\,\tau\,\la+\ka-1-2\,W}-\ka\,\tau\abs{\Om}}{\ka\,\tau\abs{\Om}}\,.
\end{align*}
The assertions now follow immediately since \eqref{Lem:TW:Eqn4} implies that $W$ and $\eps$ scale for small $\tau$ like $\tau\,\ln\at{1/\tau}$ and $1/\ln\at{1/\tau}$, respectively.
\end{proof}
Lemma \ref{Lem:Unstable} guarantees the existence of the asymptotic eigenvalues 
\begin{align*}
\la\approx\frac{1-\ka}{\;\D\tau\,\ka\,\ln\at{\frac{1}{\tau}}\;}\,\Bat{\ln\bat{2/\delta}+\iu\,\bat{1+2\,\pi\,\Zset}}
\end{align*}
and this implies the following result.
\begin{corollary}
\label{Cor:Instab}
Let $\ka\in\oointerval{0}{1}$, $\Om\neq0$ and $0<\delta<2$ be fixed. Then the traveling wave from Theorem~\ref{Thm:TW} is unstable for all small $\tau$.
\end{corollary}
However, our asymptotic analysis with respect to the $\mu$ variable does not describe the \mbox{complete} stability picture. First, in the case of $\delta<2$  it does not exclude the existence of unstable \mbox{eigenvalues} that correspond to large values of $\mu$. Secondly, numerical simulations of the dynamical problem \eqref{MicroDynamics}+\eqref{MicroMultiplier}+\eqref{DefTheta}+\eqref{DefNonl} as well as plots of $S_+$ as shown in Figures \ref{Fig:SpectrumTW1} reveal that traveling waves with $\delta<2$ can be stable as long as $\tau$ is not too small.
\appendix
%
%-------------------------------------------------------------------------------------------
\subsection*{List of symbols}
%-------------------------------------------------------------------------------------------
%
%
\begin{flushleft}
\scriptsize
\begin{tabular}{clclcl}
&$x\pair{t}{p}$&&state of the particle system &&Equation \eqref{MicroDynamics}
\\%
&$\ell\at{t}$&&prescribed dynamical constraint &&Equation \eqref{MicroConstraint}
\\%
&$\si\at{t}$&&Lagrian multiplier, non-local mean-field&&Equation \eqref{MicroMultiplier}
\\%
&$\tau$&&small relaxation time&&Equation \eqref{MicroDynamics}
\\%
&$\kappa$&&parameter for the trilinear function $H^\prime$, half width of spinodal region &&Equation \eqref{DefNonl} and Figure \ref{Fig:Trilinear}
\\%
&$\delta$&&strength of the \BMHC inhomogeneities \EMHC &&Equation \eqref{DefTheta}
\\%
&$\xi_-\at{t},\,\xi_+\at{t}$&&interface position in the particle model with $\tau>0$&&Equation \eqref{LimStates} and Figure \ref{Fig:GenericExamples2}
\\%
&$\xi\at{t}$&&interface in the limit model with $\tau=0$&&Equation \eqref{DefPhases}
and Figure \ref{Fig:LimitStates}
\medskip \\%
&$X$&&profile of a traveling wave&&Theorem \ref{Thm:TW}
\\%
&$\Om$&&speed of a traveling wave&&Theorem \ref{Thm:TW}
\\%
&$\Xi_\pm$&&interface positions in a traveling wave&&Theorem \ref{Thm:TW} and Figure \ref{Fig:TW}
\\%
&$P$&&analogue to $p$ in a comoving frame&&Equation \eqref{Lem:TW:Eqn0}
\\%
&$W$&&half of the interface width in a traveling wave&&Equations \eqref{Lem:TW:Eqn4} and \eqref{DefWidth}
\medskip \\%
&$Z$ &&eigenfunction of the linearized equation in the comoving frame&&Equations \eqref{LinStab.EP} and \eqref{EVTrafo}
\\%
& $\la$ && eigenvalue corresponding to $Z$, appears usually as $\tau\,\la$ && Equation \eqref{LinStab.EP}
\\%
&$\zeta$ &&mean-field in the spectral analysis&&Equation \eqref{LinStab.Nonlocal}
\\%
\end{tabular}
\end{flushleft}
%
%
%-------------------------------------------------------------------------------------------
\subsection*{Acknowledgements}
%-------------------------------------------------------------------------------------------
%
This work has been supported by the German Research Foundation (DFG) within the  Collaborative Research Center SFB 1060 and by the individual grant HE 6853/3-1.
%
%
% -----------------------------------------------------------------------------
% - Bibliography
% -----------------------------------------------------------------------------
%

%
%
\end{document}